\documentclass[11pt]{article}

\usepackage[utf8]{inputenc}
\usepackage[english]{babel}
\usepackage{amsmath,amssymb,amsfonts,amsthm}
\usepackage{multirow}
\usepackage{xcolor}
\usepackage{hyperref}
\hypersetup{colorlinks=true,linkcolor=blue,filecolor=blue,citecolor = blue,urlcolor=blue} 

\usepackage{graphicx}
\usepackage[scale=0.75]{geometry}

\theoremstyle{plain}
\newtheorem{theorem}{Theorem}
\newtheorem{lemma}{Lemma}
\newtheorem{proposition}{Proposition}

\theoremstyle{definition}

\theoremstyle{remark}
\newtheorem*{remark}{Remark}

\newcommand{\RR}{\mathbb{R}}

\newcommand{\di}{\text{d}}

\newcommand{\lune}{\mathcal{L}}
\newcommand{\ind}[1]{{\mathbf{1}}_{#1}}

\title{Siblings in $d$-dimensional nearest neighbour trees}
\author{Jérôme Casse\\{\small Université Paris-Saclay, CNRS}\\{\small Laboratoire de mathématiques d’Orsay}\\{\small 91405 Orsay, France}}
\date{}

\begin{document}
\maketitle

\begin{abstract}
  Pick a sequence of uniform points on the $d$-dimensional sphere. Then, link the $n$th point to its closest one that arrives in the past. This constructs a labelled tree called the nearest neighbour tree on the $d$-dimensional sphere. These trees share some properties with the random recursive tree: the height of the last arrival node, the mean degree of the root, etc. On the contrary, the number of leaves seems to depend on dimension $d$, but no such properties have been proved yet. In this article, we prove that the mean number of siblings depends on~$d$.\par
  In particular, we give explicit calculations of this number. In dimension $1$, it is $1 + \ln 2$ and, in any dimension $d$, it has an explicit integral form, but unfortunately, it does not give an explicit number. Nevertheless, we show that it converges to~$2$ when $d \to \infty$ exponentially quick at a rate of $\sqrt{3}/2$.\par
  To prove these results, we look at the local limit of those trees and we do some fine computations about the intersection of two balls in dimension $d$. In particular, we obtain a non-trivial upper bound for those intersections in some precise cases.\par
\end{abstract}

\section{Introduction}
\paragraph{Nearest-neighbour trees (NNT).}
We study an embedded version of random growing trees whose attachment rule is based on a dynamical nearest neighbour.\par
More precisely, let $(E, \delta, \mu)$ be a Polish space equipped with a metric $\delta$ and with a probability measure $\mu$ on $E$. Let $(X_n)_{n \geq 1}$ be a sequence of i.i.d.\ points sampled according to $\mu$. From this sequence, we define an increasing sequence of labelled random trees $(T_n)_{n \geq 1}$, called \emph{labelled nearest-neighbour trees} (labelled NNT), whose sets of vertices is $\{1,\dots,n\}$ and sets of edges $E_n$ is defined in the following way:
\begin{itemize}
\item $E_1 = \emptyset$,
\item for any $n \geq 2$, we denote
  \begin{equation} \label{eq:antecedent}
    A(n) = \text{argmin} \{1 \leq i \leq n-1 : \delta(X_i,X_n)\},
  \end{equation}
  then
  \begin{displaymath}
    E_{n} = E_{n-1} \cup \left\{ \left(A(n),n \right) \right\}.
  \end{displaymath}
  If the argmin is not unique, the label $A(n)$ is chosen uniformly at random between the set of labels that realise the minimum. In our step, with probability one, it never happens.
\end{itemize}
In words, at each step, the tree grows by linking the new node~$n$ to the node~$i<n$ that corresponds to the label of the point~$X_i$ that is the nearest of~$X_n$. A realisation of such a construction is given in Figure~\ref{fig:construction}. Moreover, from these labelled NNT, we can construct \emph{unlabelled rooted NNT} by forgetting the labels and rooting them at the node anciently labelled $1$, and also \emph{unlabelled unrooted NNT} by forgetting all the labels.\par

\begin{figure}
  \begin{center}
    \begin{tabular}{ccc}
      \includegraphics[width=0.4\linewidth]{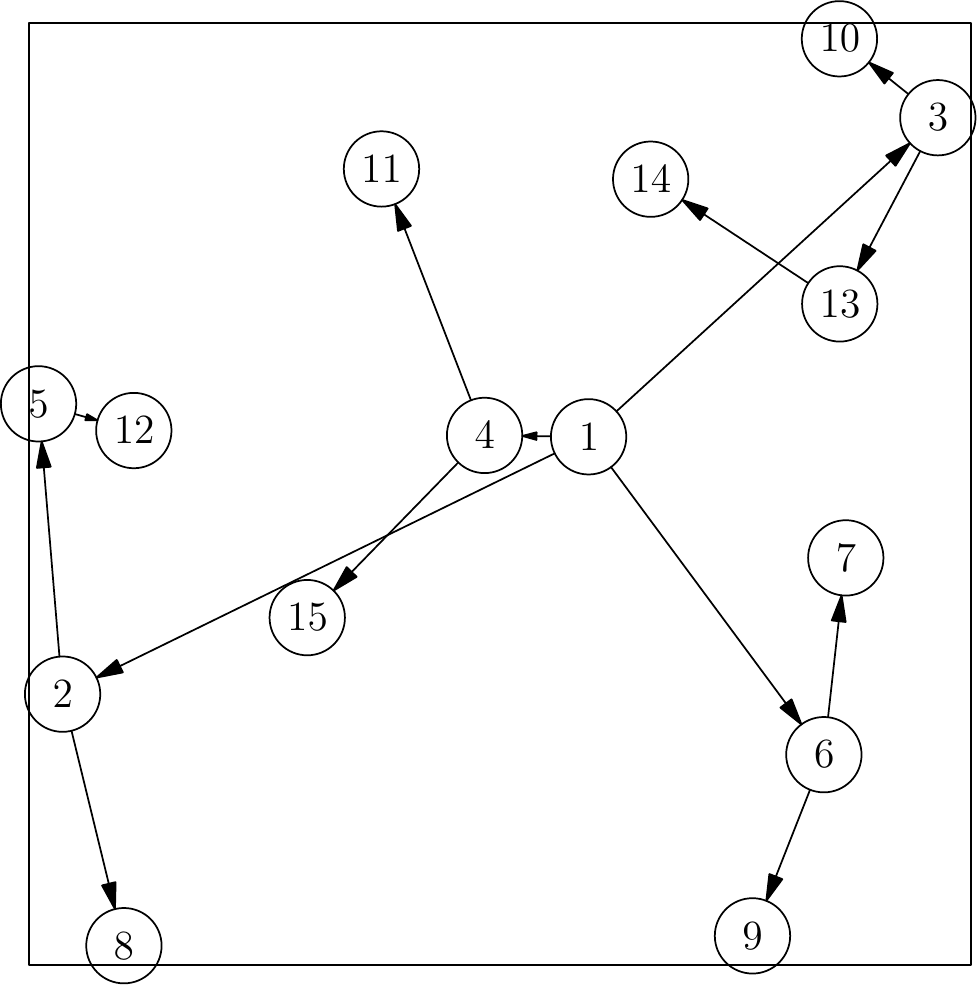} & \hspace{1cm} & \includegraphics[width=0.4\linewidth]{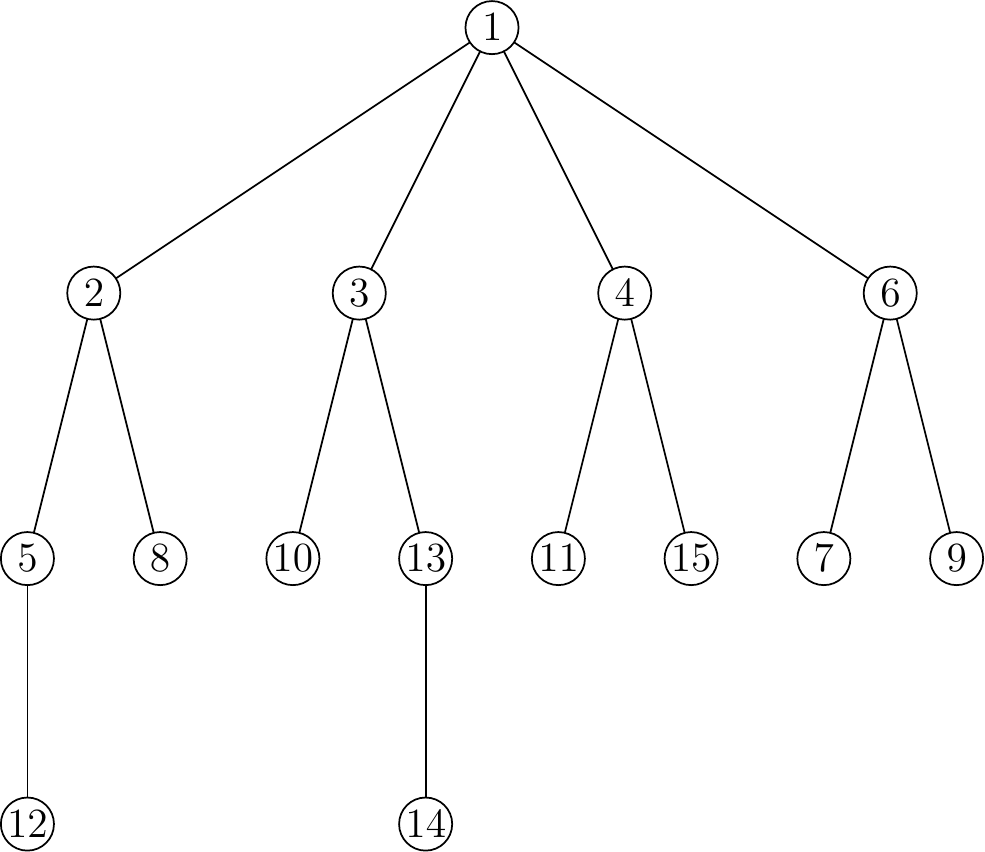}\\
      \includegraphics[width=0.4\linewidth]{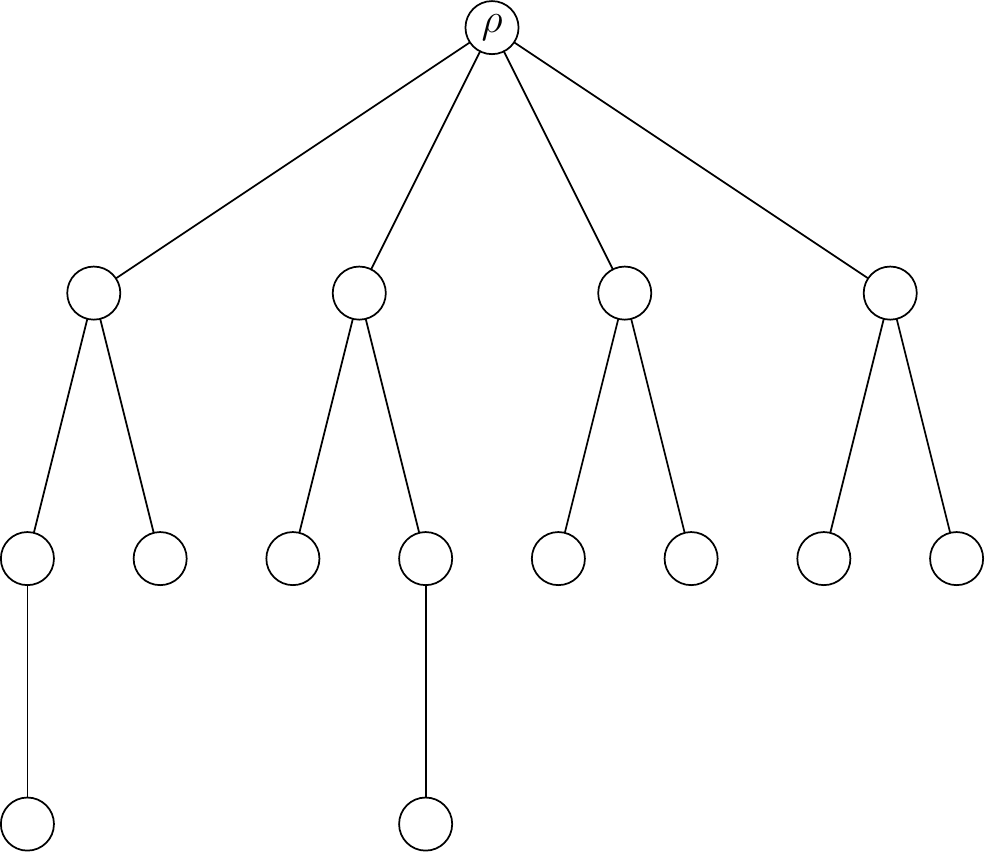} & \hspace{1cm} & \includegraphics[width=0.4\linewidth]{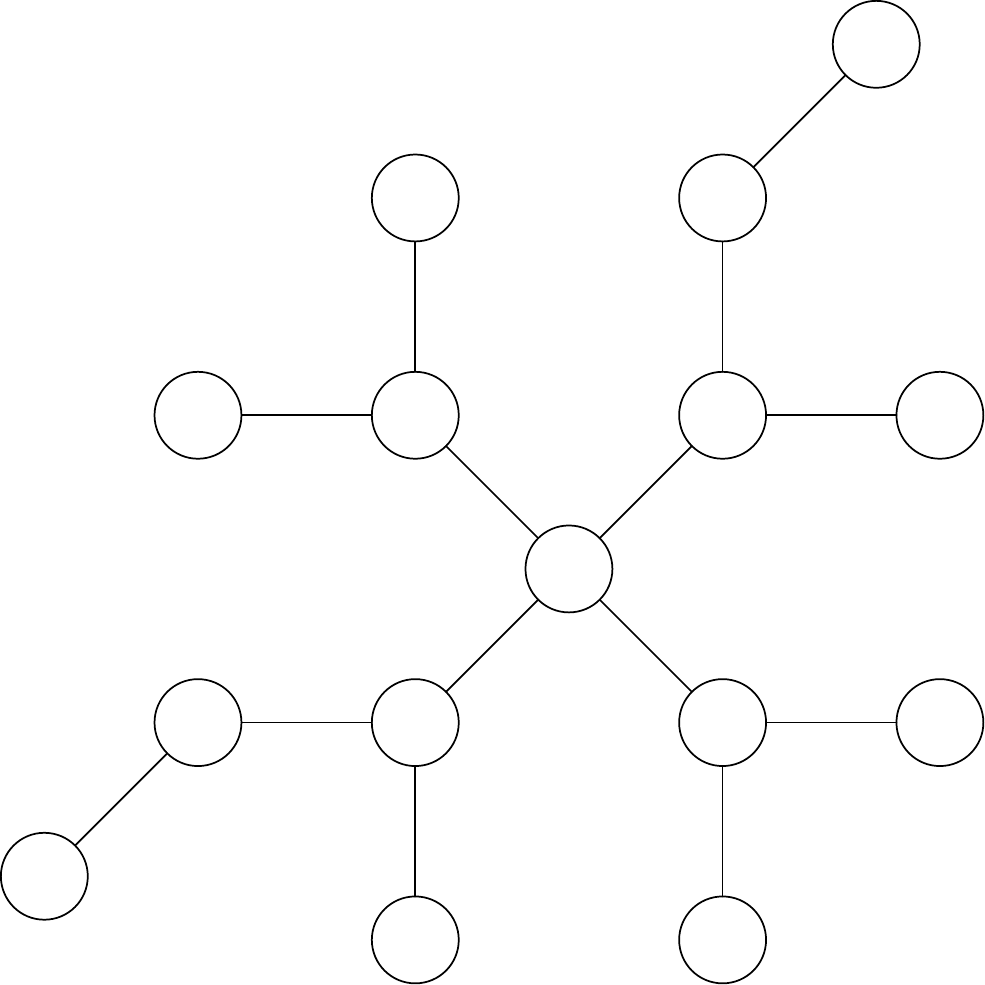}
    \end{tabular}
  \end{center}
  \caption{Top-left: Illustration of the dynamic construction of the nearest-neighbour tree of size $15$ in the square equipped with the Euclidean metric and the uniform law. The points are located in the centre of each circle. The label inside a circle indicated its order of arrivals: label $1$ is for the first point, etc. Top-right: The labelled NNT obtained by the construction. It is the main object of study of this article. Bottom-left: The unlabelled rooted version. Bottom-right: The unlabelled unrooted version.}
  \label{fig:construction}
\end{figure}

\paragraph{Random Recursive Tree}
If the chosen metric $\delta = 1$ for any couple of points, $A(n)$ is chosen uniformly in $\{1,\dots,n-1\}$ then the NNT is the Random Recursive Tree (RRT). Many properties of this tree have been studied such as its height~\cite{Devroye87,Devroye88,Pittel94,DFF10,PS22}, the degrees of its nodes~\cite{DL95,GS02}, the sizes of its subtrees~\cite{GM05,BB14}, and other properties~\cite{Mahmoud91,Dobrow96}. In particular, let us mention the survey about them~\cite{SM95} as well as chapters of some books~\cite[chapter 6]{Drmota09},~\cite[Chapter16.2]{FK23} and~\cite[chapter~9]{Curien23}. This list of references is obviously non-exhaustive and interested readers could also refer to references inside the references given.

\paragraph{d-NNT}
In this article, we focus on the cases where $(E,\delta,\mu)$ is the sphere $\mathbb{S}_d$ in dimension $d \in \{1,2,3,\dots\}$ equipped with the Euclidean metric $\lVert.\rVert_2$ on $\mathbb{S}_d$ and the Lebesgue measure $\lambda$ on~$\mathbb{S}_d$. Moreover, we are interested mainly in the labelled NNT. Hence, in the following of this article, we call a $d$-NNT (for $d$-nearest-neighbour tree), a labelled NNT constructed on the space $(\mathbb{S}_d,\lVert.\rVert_2,\lambda)$.\par
Those trees are a critical version of the geometric preferential attachment graph introduced and studied in~\cite{MS02,Jordan10,JW15} where the new node attaches (not deterministically) to a node that is both close to it and that has a high degree. When some parameters of this model degenerate, we obtain the NNT.\par
They are also a critical version of the FKP network model introduced and studied in~\cite{FKP02,BBBCR03} where the new node attach deterministically to the node that is both the closest one to him (for the euclidean metric) and to the root (for the graph metric) according to a deterministic trade-off function. When the trade-off function degenerates, we obtain NNT.\par
From this model of nearest neighbours, a Poissonian colouring was defined and studied when $d=2$ in~\cite{Preater09,Aldous18,BBCS23}.\par
Finally, the $d$-NNT themselves have been studied, to the best knowledge of the author, for now only by Lichev and Mitsche in~\cite{LM21} where they study many of their properties. One of their results used in this article concerns the local limit of $d$-NNT.

\begin{remark}
  When $d \to \infty$, the $d$-NNT seems to be the random recursive tree (RRT). Indeed, when the dimension~$d$ is very large and we take $n$ points with $n$ small according to $d$, then the distances between each couple of points are very close to $\sqrt{2}$. Hence, in the limit $d=\infty$, the distance between all points is $\sqrt{2}$ a.s. To conclude, we need to check the next order of precision. To the best knowledge of the author, this question is still open, but we have not spent too much time to search in detail into the literature on statistics in high dimensions where relatively closed works could have been done.
\end{remark}

\paragraph{Statistical properties of the trees.}
Let us come back for a moment to the general case $(E,\delta,\mu)$. What properties of $(E,\delta,\mu)$ can be inferred from the knowledge of the labelled NNT~$T_n$ or even from its unlabelled (rooted or unrooted) versions? Although it might be obvious that the geometry of the space has an influence on $T_n$, finding examples of properties that witness this dependence is harder than expected.\par
For example, finding a statistic that allows to distinguish the dimension $d$ for $d$-NNT was unknown. Indeed, many statistics do not depend on the dimension, for example for NNT with $n$~nodes, the law of the height of the $n$th point of the tree does not depend on $(E,\delta,\mu)$, the mean degree of the root is the $(n-1)$-th harmonic number $H_{n-1} = \sum_{i=1}^{n-1} \frac{1}{i}$, etc. In~\cite{JW15,LM21}, it is conjectured that the asymptotic number of leaves of $T_n$ allows to recover $d$. In this article, we prove that the expected number of siblings indeed permits recovering it.

\paragraph{Main results.}
The main results of this article concern the \emph{asymptotic mean number of siblings}, denoted $S_d$, of a node in a $d$-NNT. Let $T_n$ be a $d$-NNT of size $n$ and take a node $i \in \{2,\dots,n\}$ that is not $1$. The number $s(i)$ of siblings of $i$ in $T_n$ is the cardinal of the set of nodes that share the antecedent of $i$:
\begin{equation}
  s(i) = \left(\text{card} \{j \in \{2,\dots,n\} : A(j)=A(i)\}\right) - 1
\end{equation}
where $A(i)$ is the label defined in Equation~\eqref{eq:antecedent}. The mean number of sibling $S(T_n)$ of $T_n$ is
\begin{equation}
  S(T_n) = \frac{1}{n} \sum_{i \in \{2,\dots,n\}} s(i).  
\end{equation}

\begin{remark}
  We can extend $s$ such that $s(1) = 0$. In that case, $S(T_n) = \mathbb{E}[s(U)]$ where $U$ is uniform on $\{1,\dots,n\}$.
\end{remark}

The asymptotic of this random number is explicit in dimension $1$ and for the RRT:
\begin{theorem}\label{thm:1fini}
  Let $T_n$ be a $1$-NNT. The expected number of siblings of $T_n$ converges a.s.\ as $n \to \infty$ to $S_1 = 1 + \ln 2$.\par
  Let $T_n$ be a RRT. The expected number of siblings of $T_n$ converges a.s.\ as $n \to \infty$ to $S_\infty = 2$.
\end{theorem}

For dimension $1$, to the best knowledge of the author, it is a new result. The proof is given in Section~\ref{sec:1}. For the RRT case, this could be already deduced from previous works on RRT, for example from~\cite[Proposition~9.2]{Curien23}. We give a short proof of it in Section~\ref{sec:sibRRT}.\par
\smallskip
In any dimension $d$, the asymptotic mean number of siblings $S_d$ is not that much explicit, we obtain a closed integral form. Before giving it, let us introduce the following notations:
\begin{itemize}
\item we denote by $V_d$ the volume of a ball of radius $1$ in dimension $d$ for the Euclidean metric
  \begin{equation}
    V_d = \frac{\pi^{d/2}}{\Gamma(d/2+1)},
  \end{equation}
\item we define the function $F(z)$ by, for any $z>0$,
  \begin{equation}
    F(z) = \frac{1}{z} \left( \frac{1}{z} \ln(1+z) - \frac{1}{1+z} \right) \text{, and}
  \end{equation}
\item we define the function $z(\theta)$ by, for any $\theta \in [0,\pi]$,
  \begin{equation}
    z(\theta) = (2 \cos \theta)_+ = \begin{cases} 2 \cos(\theta) & \text{if } \theta \in [0,\pi/2], \\ 0 & \text{if } \theta \in [\pi/2,\pi]. \end{cases}
  \end{equation}
\end{itemize}

\begin{theorem} \label{thm:main}
  Let $T_n$ be a $d$-NNT. The expected number of siblings of $T_n$ converges a.s.\ as $n \to \infty$ to
  \begin{equation} \label{eq:siblings}
    S_d = 2\ \frac{(d-1) V_{d-1}}{V_d} \int_{0}^{\pi} \di{\theta}  \sin(\theta)^{d-2} \int_{z(\theta)}^\infty \di{z}\, z^{d-1} F\left(\frac{\lune(z,\theta)}{V_d} \right)
  \end{equation}
  where
  \begin{equation} \label{eq:lune}
    \frac{\lune(z,\theta)}{V_d} = z^d - \frac{V_{d-1}}{V_d} \left( \int_{\frac{1-z\cos \theta}{\sqrt{(1-z\cos \theta)^2 + (z \sin \theta)^2}}}^{1} (1-x^2)^{\frac{d-1}{2}}\, \di x + z^d \int_{-1}^{\frac{\cos \theta-z}{\sqrt{(\cos \theta -z)^2 + (\sin \theta)^2}}} (1-x^2)^{\frac{d-1}{2}}\, \di x\right).
  \end{equation}
  
  An other interesting formula for $S_d$ is
  \begin{align}
    S_d & = 2 - \underbrace{2\ \frac{(d-1)V_{d-1}}{dV_d} \int_{0}^{1} \di z \ \left( 1-z^2 \right)^{\frac{d-3}{2}} \left( 1 - \frac{\ln(1+2^d z^d)}{2^d z^d} \right)}_{T_+(d)} \nonumber \\
        & \qquad + \underbrace{2\ \frac{(d-1)V_{d-1}}{dV_d} \int_0^{\pi} \di \theta \sin(\theta)^{d-2} \int_{z(\theta)^d}^\infty \di u \left( F\left( \frac{\lune(u^{1/d},\theta)}{V_d} \right) - F(u) \right)}_{T_-(d)}. \label{eq:2sd}
  \end{align}
\end{theorem}

\begin{remark}
  The quantity $V_d z^d - \lune(z,\theta)$ is the volume of the intersection of two balls in dimension~$d$, one of radius~$1$ and the other of radius~$z$, at distance $\sqrt{1+z^2-2 z \cos \theta}$, see Figure~\ref{fig:intersection}. In Section~\ref{sec:bound}, Lemma~\ref{lem:llb} gives a non-trivial upper bound of its volume according to $z$ and $\theta$. 
\end{remark}

\begin{figure}
  \begin{center}
    \includegraphics{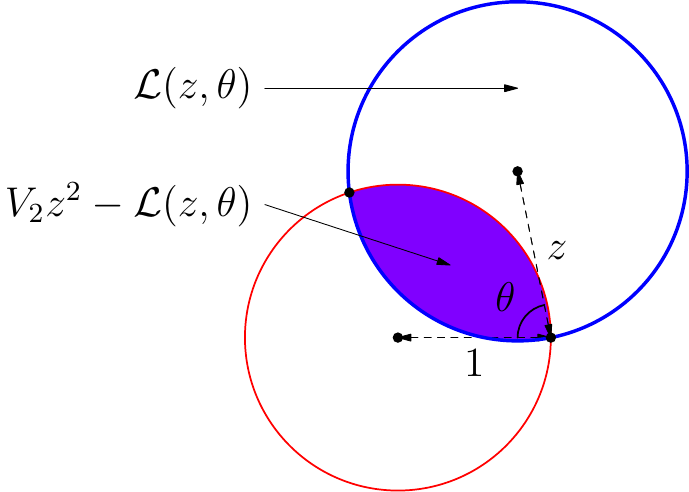}
  \end{center}
  \caption{The red circle is of radius $1$, the blue one is of radius $z$ and the distance between centres is $\sqrt{1+z^2-2 z \cos \theta}$. The purple region denotes their intersection whose volume is $V_2 z^2 - \lune(z,\theta)$ and the area of the blue ball minus the purple region/intersection is $\lune(z,\theta)$.} \label{fig:intersection}
\end{figure}

The expression~\eqref{eq:2sd} given in Theorem~\ref{thm:main} permits to prove that the mean number of siblings $S_d$ converges (conjecture: is monotonically increasing?) to~$2$ at an exponential speed according to~$d$ thanks to the following theorem that expresses the asymptotic behaviours of $T_+(d)$ and $T_-(d)$.
\begin{theorem} \label{thm:as}
  We consider the quantities $T_+(d)$ and $T_-(d)$ as defined in Theorem~\ref{thm:main}. They are both positive, and, as $d \to \infty $,
  \begin{equation}  \label{eq:asT+}
    T_+(d) \sim \frac{2 \sqrt{2\pi}}{3} \frac{1}{\sqrt{d}} \left( \frac{\sqrt{3}}{2} \right)^d
  \end{equation}
  and
  \begin{equation}
    T_-(d) = O\left(\left(\frac{4\sqrt{3}}{9} \right)^d\right). \label{eq:asT-}
  \end{equation}
\end{theorem}

\begin{remark}
  We expect that the asymptotic exponential order of $T_-$ is the good one, in the sense that, $T_-(d)$ should be a  $\Omega \left( \frac{1}{d} \left(\frac{4\sqrt{3}}{9} \right)^d \right)$.
\end{remark}

\paragraph{For unlabelled trees?}
The same statistic works for unlabelled rooted NNT. Indeed, with the knowledge of the root, we can find the antecedent of a node and so we can compute the mean number of siblings of the tree.\par
For unlabelled unrooted NNT, the statistic we can compute is the mean of the square of degrees of nodes. Indeed, for any labelled tree $T_n$,
\begin{equation}
  \frac{1}{n} \sum_{i=1}^n \text{deg}(i)^2 = \frac{1}{n} \left(\sum_{i=2}^n s(i) + 4 (n-1) - 2 \,\text{deg}(1)  \right).
\end{equation}
As $\deg(1) = O(\log(n))$ when $n \to \infty$ see~\cite[Theorem~1.6]{LM21}, the asymptotic mean of the square of degrees converges to $S_d + 4$. This statistic does not need the knowledge of the labels, nor of the root and so can be computed on unlabelled unrooted NNT.

\paragraph{Content:} In Section~\ref{sec:1fini}, we prove Theorem~\ref{thm:1fini}. We recall first the computation in the case of RRT and then we compute the case $d=1$ by recalling the local limit of the $d$-NNT. In Section~\ref{sec:dim}, we prove Theorem~\ref{thm:main}. In Sections~\ref{sec:asT+} and~\ref{sec:asT-}, we prove Theorem~\ref{thm:as} by showing respectively the asymptotic behaviours of $T_+$ and $T_-$.

\section{Siblings in RRT and in $1$-NNT} \label{sec:1fini}
\subsection{Siblings in RRT} \label{sec:sibRRT}
Let $n$ be any integer number. Let $i$ be a uniform node in $\{1,\dots,n\}$, the asymptotic mean number of siblings in the RRT is
\begin{align*}
  S_\infty = \lim_{n \to \infty} \sum_{i=2}^n \underbrace{\frac{1}{n}}_{\text{\footnotesize $i$ is uniformly chosen}} \sum_{j=1}^{i-1} \underbrace{\frac{1}{i-1}}_{\text{\footnotesize $j=A(i)$ is uniform on $\{1,\dots,i-1\}$}} \sum_{k=j+1}^n \underbrace{\frac{1}{k-1} \ind{k \neq i}.}_{\underset{\text{\footnotesize to be a child of $j$}}{\text{\footnotesize probability for $k>j$ and $k \neq i$}}}
\end{align*}
Normalised in $n$ that gives
\begin{displaymath}
  S_\infty = \int_{0}^1 \di x \int_{0}^{x} \frac{1}{x} \di y \int_{y}^1 \frac{1}{z} \di z = - \int_0^1 \frac{1}{x} \di x \int_0^x \di y \ln(y) = \int_0^1 (1-\ln(x)) \di x = 2.
\end{displaymath}

\subsection{Local limit of $d$-NNT} \label{sec:ll}
In~\cite[Theorem~1.10]{LM21}, Lichev and Mitsche proved that a $d$-NNT re-rooted at a uniform random vertex admits a local limit (this is the so-called Benjamini--Schramm limit) which is a random pointed infinite one-ended tree, called Poisson $d$-NN random tree. The Poisson $d$-NN random tree is defined in the following: let $\mathcal{P}$ be a Poisson Point Process in $\RR^d$ and sample $(X_v)_{v \in \mathcal{P}}$ a collection of independent uniform random variables on $[0,1]$ (considered as arrival times), then connect each point $v \in \mathcal{P}$ to its closest older neighbour $A(v) = \text{argmin} \{u \in \mathcal{P} \cap \{u : X_u < X_v\}  : \delta(X_u,X_v)\}$. In particular, we have concentration and convergence of all local statistics. A drawing of the local limit in dimension $2$ is given in Figure~\ref{fig:ll2}.
\begin{figure}
  \begin{center}
    \includegraphics[scale=0.8]{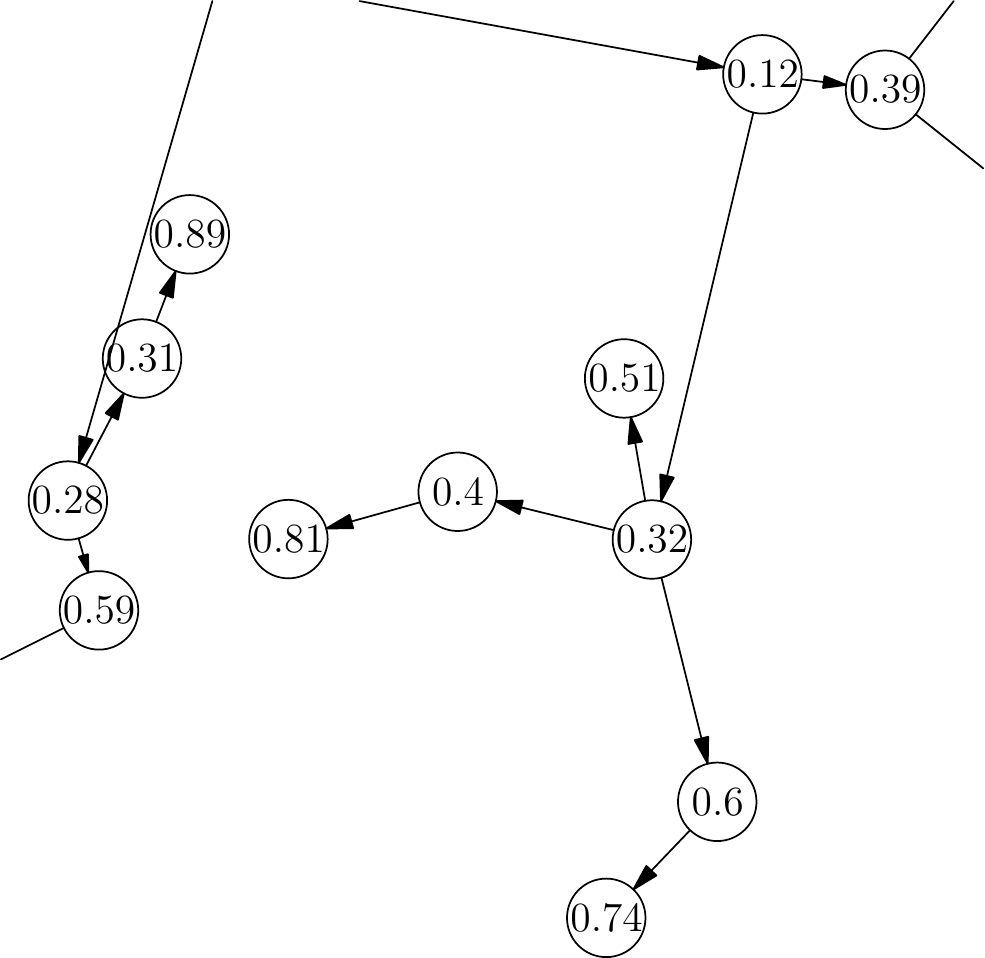}
  \end{center}
  \caption{Local limit of the $2$-NNT. Each circle represents a point of the PPP. In each circle, its random arrival time $X_v$ is indicated. Arrows go from $A(v)$ to $v$.} \label{fig:ll2}
\end{figure}

\subsection{Siblings in dimension $1$} \label{sec:1}
Thanks to the local limit defined in~\ref{sec:ll}, see also Figure~\ref{fig:sibling1}, the limit of the expected number of siblings is
\begin{align}
  S_1 & = \int_{0}^1 \underbrace{\di u_1}_{\text{label of $u_1$ is uniform}} \int_{0}^{\infty} \underbrace{\di r\ 2 u_1 e^{-2 u_1 r}}_{\text{distance between $u_1$ and $A(u_1)$}} \int_{0}^{u_1} \underbrace{\frac{1}{u_1} \di v}_{\text{label of $A(u_1)$ is uniform in $[0,u_1]$}}\\
      & \qquad \underbrace{\left(\underbrace{\int_{-\infty}^{-r} \di x \int_{v}^{u_1} \di u_2 e^{2 x u_2}}_{\text{blue area}} + \underbrace{\int_{r/2}^{r} \di x \int_{u_1}^{1} \di u_2 e^{-2(r-x)(u_2-u_1)}}_{\text{red area}} + \underbrace{\int_{r}^\infty \di x \int_{v}^{1} \di u_2 e^{-2(x-r)u_2}}_{\text{green area}}\right)}_{\text{expected number of children of $A(u_1)$ without $u_1$}}.
\end{align}

Let take some times to explain how we find the integral for the colored areas:
\begin{itemize}
\item A point $(x,u_2)$ with $x \leq -r$ (in the blue area of Figure~\ref{fig:sibling1}) is a child of $(r,v)$ if its label $u_2 \in [v,u_1]$. Indeed, $u_2$ must be greater than $v$ to be a child of $(r,v)$ and lesser than $u_1$ to not be a child of $(0,u_1)$. Moreover, the ball of centre $x$ that goes through $r$ must be free of points of label lesser than $u_2$. Indeed, if such a point exists, then it is the ancestor of $(x,u_2)$, and so it is not $(r,v)$; this ball is the segment $[2x-r,r]$, but we know that $[-r,r]$ is free of such points, then we just need to compute the probability that $[2x-r,-r]$ does not contain a point of label less than $u_2$. As the repartition of points is a PPP, this probability values $e^{2xu_2}$.
\item A point $(x,u_2)$ with $r/2 \leq x \leq r$ (in the red area of Figure~\ref{fig:sibling1}) is a child of $(r,v)$ if its label $u_2 \in [u_1,1]$. Indeed, $u_2$ must be greater than $v$ to be a child of $(r,v)$, but points in the segment $[-r,r]$ have their labels greater than $u_1 \leq v$. Moreover, the ball of centre $x$ that goes through $r$ must be free of points of label lesser than $u_2$, knowing already that it is free of points of label greater than $u_1$; this ball is the segment $[2x-r,r]$. As the repartition of points is a PPP, this probability values $e^{-2(r-x)(u_2-u_1)}$.
\item A point $(x,u_2)$ with $r \leq x$ (in the green area of Figure~\ref{fig:sibling1}) is a child of $(r,v)$ if its label $u_2 \in [v,1]$. Moreover, the ball of centre $x$ that goes through $r$ must be free of points of label lesser than $u_2$; this ball is the segment $[r,2x-r]$. As the repartition of points is a PPP, this probability values $e^{-2(x-r)u_2}$.
\end{itemize}

\begin{figure}
    \centering
    \includegraphics{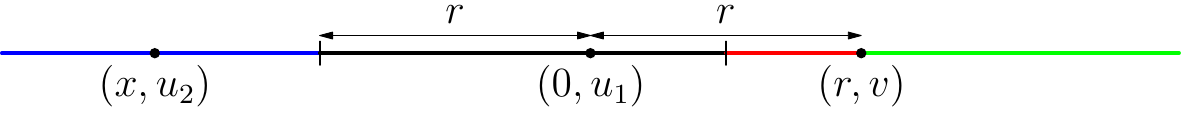} 
    \caption{All the points in $[-r,r]$ (black and red areas) have a label greater than $u_1$. Blue area: a point $(x,u_2)$ is a child of $(r,v)$ iff the label $u_2 \in [v,u_1]$ and if there does not exist a point in $[2x-r,r]$ whose label is in $[0,u_2]$, but for sure such a point can not exist in $[-r,r]$. Black area: no point can be a child of $(r,v)$. Red area: a point $(x,u_2)$ is a child of $(r,v)$ iif $u_2 \in [u_1,1]$ and there does not exist a point in $[2x-r,r]$ with a label in $[u_1,v]$. Green area: a point $(x,u_2)$ is a child of $(r,v)$ iff the label $u_2 \in [v,1]$ and there does not exist a point in $[r,2x-r]$ whose label is in $[0,u_2]$.}
    \label{fig:sibling1}
\end{figure}

Another way to compute it is to count the older siblings and multiply by $2$. We adopt this approach because it is the one used later in any dimension.
\begin{align}
  S_1 & = 2 \int_{0}^1 \di u_1  \int_{0}^{\infty} \di r\ 2 u_1 e^{-2 u_1 r} \int_{0}^{u_1} \frac{1}{u_1} \di v \left(\int_{-\infty}^{-r} \di x \int_{v}^{u_1} \di u_2 e^{2 x u_2} + \int_{r}^\infty \di s \int_{v}^{u_1} \di u_2 e^{-2(x-r)u_2}  \right).
\end{align}

To compute this, we first integrate on  $x$,
\begin{displaymath}
  S_1 = 4 \int_{0}^1 \di u_1 \int_{0}^{u_1} \di v \int_{v}^{u_1} \di u_2 \int_{0}^{\infty} \di r\ e^{-2 u_1 r} \left( \frac{e^{-2ru_2}}{2u_2} + \frac{1}{2u_2} \right).
\end{displaymath}
Then, on $r$,
\begin{displaymath}
  S_1 = 2 \int_{0}^1 \di u_1 \int_{0}^{u_1} \di v \int_{v}^{u_1} \di u_2 \frac{1}{u_2} \left( \frac{1}{2(u_1+u_2)} + \frac{1}{2u_1} \right). 
\end{displaymath}
Now, we decompose in partial fractions and we conclude by integrating successively on $u_2$, $v$ and $u_1$,
\begin{align*}
  S_1 & = \int_{0}^1 \di u_1 \frac{1}{u_1}\int_{0}^{u_1} \di v \int_{v}^{u_1} \di u_2 \, \left(\frac{2}{u_2} - \frac{1}{u_1+u_2} \right)\\
      & = \int_{0}^1 \di u_1 \frac{1}{u_1}\int_{0}^{u_1} \di v \left(2 \ln(u_1) - 2 \ln(v) - \ln(2u_1) + \ln(u_1+v) \right)\\
      & = \int_{0}^1 \di u_1 \frac{1}{u_1}\int_{0}^{u_1} \di v \left(\ln(u_1) - 2 \ln(v) - \ln(2) + \ln(u_1+v) \right)\\
      & = \int_{0}^1 \di u_1 \frac{1}{u_1} \big(u_1 \ln u_1 - 2 u_1 (\ln(u_1)- 1) - u_1 \ln(2) + u_1 (\ln(u_1) - 1 + 2 \ln 2)\big)\\
      & = \int_{0}^1 \di u_1 \left(1 + \ln 2\right) = 1+\ln(2). \qed
\end{align*}

\section{Proof of Theorem~\ref{thm:main}} \label{sec:dim}
In the following, $V_d$ denotes the volume of the unit ball of dimension $d$, i.e.\ $V_d = \frac{\pi^{d/2}}{\Gamma(d/2+1)}$ and $A_{d-1} = \frac{2 \pi^{d/2}}{\Gamma(d/2)}$ denotes the area of its surface. We recall that $A_{d-1}= d V_d$.

\paragraph{From local limit to Equation~\eqref{eq:siblings}:}
As, in dimension $1$, we use the local limit given in Section~\ref{sec:ll}, the mean number of siblings is
\begin{align*} \label{eq:sibling}
  S_d = & 2 \int_0^1 \underbrace{\di u_1}_{\text{label of $u_1$ is uniform}} \int_{0}^\infty \underbrace{\di r\, d\,V_d u_1 r^{d-1} e^{-V_d u_1 r^d}}_{\text{distance between $u_1$ and $A(u_1)$}} \int_{0}^{u_1} \underbrace{\frac{1}{u_1}\di v}_{\text{label of $A(u_1)$ is uniform in $[0,u_1]$}} \\
        & \qquad \underbrace{\int_{\RR^d \backslash B(0,r)} \di z}_{\text{no older sibling of $u_1$ could be in $B(0,r)$}} \int_{v}^{u_1} \di u_2 \underbrace{e^{-u_2 \lune_r(z)}}_{\text{probability to be an older sibling of $u_1$}}.
\end{align*}
The first $2$ comes because, in the integral above, we count only the older siblings. In addition, the probability $e^{-u_2 \lune(z_2)}$ to be an older sibling of $u_1$ for a point of label $u_2 \in [v,u_1]$ at position~$z$ is to have no point older than it in the ball whose centre is $z$ and passing from the point of label~$v$, but remember that such a point could not exist in the ball of centre $0$ and radius $r$. Hence, the quantity~$\lune_r(z)$ denotes the volume of the ball centred in $z$ and that goes through $(r,0,\dots,0)$ where we exclude the part in the ball centred in $0$ and of radius $r$:
\begin{equation}
  \lune_r(z) = \text{Vol}\left(B(z,|z-(r,0,\dots,0)|) \backslash B(0,r)\right).
\end{equation}
  
  Now, we compute this integral in any dimension $d \geq 2$. Firstly, we make a change of coordinates passing from $z \in \RR^d$ to $(s,\theta)$ where $s$ is the distance between $z$ and $(r,0,\dots,0)$ and $\theta$ the angle between the two lines $(z,(r,0,\dots,0))$ and $(0,(r,0,\dots,0))$ whose Jacobian is $s \sin(\theta)^{d-2}$ from passing Cartesian coordinates to spherical coordinates. In the same way, we circularly integrate around the $x$-axis.
  \begin{align} 
    S_d & = 2 d V_d \int_0^1 \di u_1 \int_{0}^\infty \di r\, r^{d-1} e^{-V_d u_1 r^d} \int_{0}^{u_1} \di v \int_{0}^\pi \di \theta \int_{s(\theta)}^\infty \di s\, \underbrace{s \sin(\theta)^{d-2}}_{\text{Jacobian}} \underbrace{A_{d-2} s^{d-2}}_{\text{volume of points $(s,\theta)$}}  \int_{v}^{u_1} \di u_2 e^{-u_2 \lune_r(s,\theta)} \nonumber \\
        & =  2 d A_{d-2} V_d \int_{0}^\pi \di \theta \, \sin(\theta)^{d-2} \int_0^1 \di u_1 \int_{0}^{u_1} \di v \int_{v}^{u_1} \di u_2 \int_{0}^\infty \di r\, r^{d-1} e^{-V_d u_1 r^d}   \int_{s(\theta)}^\infty \di s \, s^{d-1}  e^{-u_2 \lune_r(s,\theta)} \label{eq:sibling2}
  \end{align}
  where $s(\theta) = \begin{cases} 2 r \cos(\theta) & \text{if } \theta \in [0,\pi/2] \\ 0 & \text{if } \theta \in [\pi/2,\pi] \end{cases}$.
  
  Now, we do the change of variable $s = z r$ to look at the relative distances. Hence,
  \begin{align*}
    \int_{0}^\infty \di r\, r^{d-1} e^{-V_d u_1 r^d}   \int_{s(\theta)}^\infty \di s \, s^{d-1}  e^{-u_2 \lune_r(s,\theta)} & = \int_{0}^\infty \di r\, r^{d-1} e^{-V_d u_1 r^d} \int_{z(\theta)}^\infty \di z \, z^{d-1} r^{d}  e^{-u_2 \lune_1(z,\theta) r^d}\\ & = \int_{z\theta)}^\infty \di z\, z^{d-1} \int_{0}^\infty \di r\, r^{2d-1} e^{-(u_1 V_d + u_2 \lune_1(z,\theta)) r^d}\\ & = \frac{1}{d} \int_{z(\theta)}^\infty \di z\, z^{d-1}\frac{1}{(u_1 V_d + u_2 \lune(z,\theta))^2}
  \end{align*}
where $z(\theta) = (2 \cos \theta)_+$. In the following, $\lune$ will stay for $\lune_1$.
  
  Now, we can go on with the computation of Equation~\eqref{eq:sibling2} by integrating on $u_2$, $v$ and $u_1$,
  \begin{align*}
    & \int_0^1 \di u_1 \int_{0}^{u_1} \di v \int_{v}^{u_1} \di u_2 \frac{1}{(u_1 V_d + u_2 \lune(z,\theta))^2}\\
    & = \int_0^1 \di u_1 \int_{0}^{u_1} \di v \, \frac{u_1-v}{u_1(V_d+\lune(z,\theta))(v\lune(z,\theta) + u_1 V_d)}\\
    & = \frac{1}{V_d+\lune(z,\theta)} \int_0^1 \di u_1 \frac{1}{u_1} \, \int_{0}^{u_1} \di v \, \frac{u_1-v}{v\lune(z,\theta) + u_1 V_d}\\
    & = \frac{1}{V_d+\lune(z,\theta)} \int_0^1 \di u_1  \frac{1}{\lune(z,\theta)} \left(\frac{\lune(z,\theta)+V_d}{\lune(z,\theta)} \ln \left( \frac{\lune(z,\theta)+V_d}{V_d} \right) - 1 \right)\\
    & = \frac{1}{\lune(z,\theta)^2} \ln \left( \frac{\lune(z,\theta)+V_d}{V_d} \right) - \frac{1}{\lune(z,\theta) (\lune(z,\theta)+V_d)} \\
    & = \frac{1}{V_d^2} \frac{1}{\frac{\lune(z,\theta)}{V_d}} \left( \frac{1}{\frac{\lune(z,\theta)}{V_d}} \ln \left( 1 + \frac{\lune(z,\theta)}{V_d} \right) - \frac{1}{1 + \frac{\lune(z,\theta)}{V_d}} \right).
  \end{align*}
  Hence,
  \begin{displaymath}
    S_d = 2 \frac{(d-1) V_{d-1}}{V_d}\int_{0}^\pi \di \theta \, \sin(\theta)^{d-2} \int_{z(\theta)}^\infty \di z\, z^{d-1} \frac{1}{\frac{\lune(z,\theta)}{V_d}} \left( \frac{1}{\frac{\lune(z,\theta)}{V_d}} \ln \left( 1 + \frac{\lune(z,\theta)}{V_d} \right) - \frac{1}{1 + \frac{\lune(z,\theta)}{V_d}} \right),
  \end{displaymath}
  that is Equation~\eqref{eq:siblings} of Theorem~\ref{thm:main}.
  
  \paragraph{Exact computation of $\lune(z,\theta)$:}
  The volume of the intersection of two balls, one of radius $1$ and the other of radius $z$ where centres are at distance $\delta$ (see Figure~\ref{fig:inter}) is, in dimension $d$,
  \begin{displaymath}
    I = \int_{\frac{1-z^2+\delta^2}{2 \delta}}^{1} \di x \, V_{d-1} (1-x^2)^{(d-1)/2} + \int_{-z}^{\frac{1-z^2-\delta^2}{2 \delta}} \di x \, V_{d-1} (z^2-x^2)^{(d-1)/2}.
  \end{displaymath}

  \begin{figure}
    \begin{center}
      \includegraphics{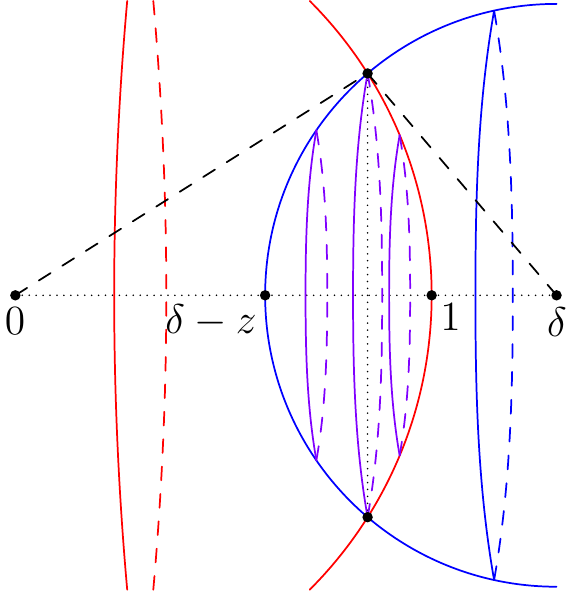}
    \end{center}
    \caption{The drawing when $d=3$. In red, the ball of radius $1$; in blue, the one of radius $z$. In purple, we denote the intersection. The volume $I$ of the intersection is obtained by integrating the volume of the purple disks, that become $(d-1)$-dimensional balls for the intersection of two $d$-dimensional balls.} \label{fig:inter}
  \end{figure}
  
  In our cases, $\delta = \sqrt{1+z^2 -2 z \cos \theta}$ by the law of cosines. Remarking that $\frac{\lune(z,\theta)}{V_d} = \frac{V_d z^d - I}{V_d}$, it gives
  \begin{align*}
    \frac{\lune(z,\theta)}{V_d} & = z^d - \frac{V_{d-1}}{V_d} \left( \int_{\frac{1-z^2+(1+z^2-2z \cos\theta)}{2 \sqrt{1+z^2 -2 z \cos \theta}}}^{1} (1-x^2)^{\frac{d-1}{2}}\, \di x + \int_{-z}^{\frac{1-z^2-(1+z^2-2z\cos\theta)}{2 \sqrt{1+z^2-2z\cos \theta}}} (z^2-x^2)^{\frac{d-1}{2}}\, \di x\right)\\
                                & =  z^d - \frac{V_{d-1}}{V_d} \left( \int_{\frac{2(1-z \cos\theta)}{2 \sqrt{1+z^2 -2 z \cos \theta}}}^{1} (1-x^2)^{\frac{d-1}{2}}\, \di x + z^{d-1} \int_{-z}^{\frac{2z(\cos(\theta) -z)}{2 \sqrt{1+z^2-2z\cos \theta}}} (1-(x/z)^2)^{\frac{d-1}{2}}\, \di x\right).
  \end{align*}
  After the change of variable $y=x/z$ in the second term of the sum, we find equation~\eqref{eq:lune}.

  \paragraph{Difference of $S_d$ with $2$:}
  First, remark than
  \begin{align*}
    2 & = \frac{2d}{\sqrt{\pi}} \frac{\Gamma(d/2)}{\Gamma((d-1)/2)} \cdot \sqrt{\pi} \frac{\Gamma((d-1)/2)}{\Gamma(d/2)} \cdot \frac{1}{d}\\
      & = 2\, \frac{(d-1)V_{d-1}}{V_d} \cdot \int_{0}^\pi \di \theta \, \sin(\theta)^{d-2} \cdot \int_{0}^\infty \di z \, z^{d-1} F\left( z^d \right).
  \end{align*}
  
  Hence, the difference $2-S_d$ could be written as the difference of two positive terms
  \begin{align*}
    2-S_d & = \underbrace{2\ \frac{(d-1)V_{d-1}}{V_d} \int_{0}^{\pi/2} \di \theta \, \sin(\theta)^{d-2} \int_{0}^{2 \cos \theta} \di z\, z^{d-1} F(z^d)}_{T_+(d)} \\
          & \quad - \underbrace{2\ \frac{(d-1)V_{d-1}}{V_d} \int_{0}^{\pi} \di \theta \sin(\theta)^{d-2} \int_{z(\theta)}^\infty \di z \, z^{d-1} \left(F\left(\frac{\lune(z,\theta)}{V_d}\right)-F(z^d)\right)}_{T_{-}(d)}.
  \end{align*}
  
  Now, we find an alternative written of $T_+(d)$. We start with the change of variable $u = z^d$ to simplify $T_+(d)$:
  \begin{align*}
    T_+(d) & = 2\ \frac{(d-1)V_{d-1}}{d V_d} \int_{0}^{\pi/2} \di \theta \, \sin(\theta)^{d-2} \int_{0}^{2^d (\cos \theta)^d} \di u \, F(u)\\
           & = 2\ \frac{(d-1)V_{d-1}}{d V_d} \int_{0}^{\pi/2} \di \theta \, \sin(\theta)^{d-2} \left[\frac{-\ln(1+u)}{u} \right]_0^{2^d (\cos \theta)^d} \\
           & = 2\ \frac{(d-1)V_{d-1}}{d V_d}  \int_{0}^{\pi/2} \di \theta \, \sin(\theta)^{d-2} \left(1 - \frac{\ln(1+2^d (\cos \theta)^d)}{2^d (\cos \theta)^d}\right)
  \end{align*}\
  and a second change of variable $z = \cos \theta$ to obtain
  \begin{displaymath}
    T_+(d) = 2\ \frac{(d-1)V_{d-1}}{d V_d} \int_0^1 \di z \, \left(1-z^2\right)^{\frac{d-3}{2}} \left(1 - \frac{\ln(1+2^d z^d)}{2^d z^d}\right). \qed
  \end{displaymath}

\section{Asymptotic behaviour of $T_+$ (proof of Equation~\eqref{eq:asT+} in Theorem~\ref{thm:as})} \label{sec:asT+}
To study the asymptotic of $T_+$, we just study the one of
\begin{itemize}
\item $\displaystyle 2 \frac{(d-1)V_{d-1}}{d V_d} = \frac{2d}{\sqrt{\pi}} \frac{\Gamma(d/2)}{\Gamma((d-1)/2)}$ whose asymptotic is $\sqrt{\frac{2d}{\pi}} + O(d^{-1/2})$,
\item $\displaystyle \int_0^{1/2} \di z \, \left(1-z^2\right)^{\frac{d-3}{2}} \left(1 - \frac{\ln(1+2^d z^d)}{2^d z^d} \right)$, and
\item $\displaystyle \int_{1/2}^{1} \di z \, \left(1-z^2\right)^{\frac{d-3}{2}} \left(1 - \frac{\ln(1+2^d z^d)}{2^d z^d} \right)$.
\end{itemize}

Asymptotics of both last integrals are given in the following lemma.
\begin{lemma} \label{lem:asT+}
  As $d \to \infty$,
  \begin{equation} \label{eq:0d}
    \int_0^{1/2} \di z \, \left(1-z^2\right)^{\frac{d-3}{2}} \left(1 - \frac{\ln(1+2^d z^d)}{2^d z^d}\right) \sim \frac{\pi \sqrt{3} - 3}{8d} \left(\frac{\sqrt{3}}{2} \right)^{d-3}
  \end{equation}
  and
  \begin{equation} \label{eq:d1}
    \int_{1/2}^{1} \di z \, \left(1-z^2\right)^{\frac{d-3}{2}} \left(1 - \frac{\ln(1+2^d z^d)}{2^d z^d}\right) \sim \frac{\pi\sqrt{3}+3}{8d}  \left(\frac{\sqrt{3}}{2} \right)^{d-3}.
  \end{equation}
\end{lemma}

This lemma is proved in the two following sections. Hence, putting all together, we obtain
\begin{displaymath}
  T_+(d) \sim \sqrt{\frac{2d}{\pi}} \frac{1}{d} \left(\frac{\sqrt{3}}{2} \right)^{d-3} \left(\frac{\pi\sqrt{3}-3}{8} + \frac{\pi\sqrt{3}+3}{8} \right)= \frac{1}{\sqrt{d}}\frac{\sqrt{6 \pi}}{4} \left(\frac{2}{\sqrt{3}} \right)^3 \left(\frac{\sqrt{3}}{2} \right)^{d}
\end{displaymath}
that is Equation~\eqref{eq:asT+}.

\subsection{When $z \in  (0,1/2)$ (proof of Equation~\eqref{eq:0d})}
  We do the change of variable $u = 1/2 -z$ to get the issue around $0$ and not $1/2$:
  \begin{align*}
    & \int_0^{1/2} \di z \, \left(1-z^2\right)^{\frac{d-3}{2}} \left(1 - \frac{\ln(1+2^d z^d)}{2^d z^d}\right)\\
    & = \int_0^{1/2} \di u \, \left(\frac{3}{4} +u -u^2\right)^{\frac{d-3}{2}} \left(1 - \frac{\ln(1+2^d (1/2-u)^d)}{2^d (1/2-u)^d}\right)\\
    & = \left(\frac{\sqrt{3}}{2}\right)^{d-3} \int_0^{1/2} \di u \, \left(1 + \frac{4}{3} u - \frac{4}{3} u^2 \right)^{\frac{d-3}{2}} \left(1 - \frac{\ln(1+(1-2u)^d)}{(1-2u)^d}\right).
  \end{align*}
  Now, we do the change of variable $u=\frac{x}{d}$, hence
  \begin{align*}
    & \int_0^{1/2} \di u \, \left(1 + \frac{4}{3} u - \frac{4}{3} u^2 \right)^{\frac{d-3}{2}} \left(1 - \frac{\ln(1+(1-2u)^d)}{(1-2u)^d}\right) \\
    & = \frac{1}{d} \int_0^{d/2} \di x \, \left(1 + \frac{4}{3} \frac{x}{d} - \frac{4}{3} \frac{x^2}{d^2} \right)^{\frac{d-3}{2}} \left(1 - \frac{\ln(1+\left(1-\frac{2x}{d}\right)^d)}{\left(1-\frac{2x}{d}\right)^d}\right).
  \end{align*}
  When $d \to \infty$,
  \begin{align*}
    & \int_0^{d/2} \di x \, \left( 1 + \frac{2x}{d} \right)^{\frac{d-3}{2}} \left( 1 - \frac{2x}{3d} \right)^{\frac{d-3}{2}} \left(1 - \frac{\ln(1+(1-\frac{2x}{d})^d)}{(1-\frac{2x}{d})^d}\right) \\
    & \to \int_0^{\infty} \di x \, e^{x}\, e^{-x/3} \left(1 - \frac{\ln(1+e^{-2x})}{e^{-2x}}\right) \text{(Now: $y = e^{-2x/3}$)} \\
    & = \frac{3}{2} \int_{0}^1 \di y\, \frac{1}{y^2} \left(1-\frac{\ln(1+y^3)}{y^3} \right) =  \frac{\pi \sqrt{3} - 3}{8}. 
  \end{align*}
  The computation of the last integral is obtained by a standard computer algebra system. 

  \subsection{When $z \in  (1/2,1)$ (proof of Equation~\eqref{eq:d1})}
  It is very similar to the case when $z \in (0,1/2)$. We do the change of variable $u = z - 1/2$ to get the issue around $0$ and not $1/2$:
  \begin{align*}
    & \int_{1/2}^{1} \di z \, \left(1-z^2\right)^{\frac{d-3}{2}} \left(1 - \frac{\ln(1+2^d z^d)}{2^d z^d}\right)\\
    & = \left(\frac{\sqrt{3}}{2}\right)^{d-3} \int_0^{1/2} \di u \, \left(1 - \frac{4}{3} u - \frac{4}{3} u^2 \right)^{\frac{d-3}{2}} \left(1 - \frac{\ln(1+(1+2u)^d)}{(1+2u)^d}\right)\\
    & = \left(\frac{\sqrt{3}}{2}\right)^{d-3} \int_0^{1/2} \di u \, \left(1 - \frac{4}{3} u - \frac{4}{3} u^2 \right)^{\frac{d-3}{2}} \left(1 - \frac{\ln((1+2u)^d) + \ln \left(1+(1+2u)^{-d}\right)}{(1+2u)^d} \right).
  \end{align*}
  Now, we do the change of variable $u=\frac{x}{d}$, hence we obtain
  \begin{align*}
    \left(\frac{\sqrt{3}}{2}\right)^{d-3}  \frac{1}{d} \int_0^{d/2} \di x \, \left(1 - \frac{4}{3} \frac{x}{d} - \frac{4}{3} \frac{x}{d}^2 \right)^{\frac{d-3}{2}} \left(1 - \frac{\ln \left( \left(1+\frac{2x}{d}\right)^d\right) + \ln \left(1+\left(1+\frac{2x}{d}\right)^{-d} \right)}{\left(1+\frac{2x}{d}\right)^d} \right).
  \end{align*}
  When $d \to \infty$,
  \begin{align*}
    & \int_0^{d/2} \di x \, \left( 1 - \frac{2x}{d} \right)^{\frac{d-3}{2}} \left( 1 + \frac{2x}{3d} \right)^{\frac{d-3}{2}} \left(1 - \frac{\ln \left( \left(1+\frac{2x}{d}\right)^d\right) + \ln \left(1+\left(1+\frac{2x}{d}\right)^{-d} \right)}{\left(1+\frac{2x}{d}\right)^d} \right) \\
    & \to \int_0^{\infty} \di x \, e^{-2x/3} \left(1 - 2x e^{-2x} - e^{-2x}  \ln\left(1+e^{-2x} \right) \right)  \\
    & = \frac{3}{2} \int_{0}^1 \di y\, \left(1 + 3 y^3 \ln y - y^3 \ln(1+y^3) \right) = \frac{\pi\sqrt{3}+3}{8}.
  \end{align*}

\section{Asymptotic behaviour of $T_-$ (proof of Equation~\eqref{eq:asT-} in Theorem~\ref{thm:as})} \label{sec:asT-}
In all this section, the dimension $d$ is supposed to be greater than $2$.

The exact asymptotic of $T_-$ is more complicated that the one of $T_+$. But, we just need an upper bound that is negligible according to the asymptotic of $T_+$. Let us recall that
\begin{displaymath}
 T_-(d) = 2 \frac{(d-1)V_{d-1}}{d V_d} \int_{0}^{\pi} \di \theta \sin(\theta)^{d-2} \int_{z(\theta)^d}^\infty \di u \left(F\left(\frac{\lune(u^{1/d},\theta)}{V_d}\right)-F(u)\right).
\end{displaymath}

The asymptotic of $\displaystyle 2 \frac{(d-1)V_{d-1}}{dV_d} \sim \sqrt{\frac{2d}{\pi}}$ is already known. Hence, we need to prove that
\begin{equation} \label{eq:int}
\int_{0}^{\pi} \di \theta \sin(\theta)^{d-2} \int_{z(\theta)^d}^\infty \di u \left(F\left(\frac{\lune(u^{1/d},\theta)}{V_d}\right)-F(u)\right) = O\left(\frac{1}{\sqrt{d}} \left(\frac{4\sqrt{3}}{9}\right)^d \right).
\end{equation}

To do it, we first give some upper bounds of $F\left(\frac{\lune(u^{1/d},\theta)}{V_d}\right)-F(u)$ in Proposition~\ref{prop:bound} above and, then, we split the integral into $12$ regions. For each of these $12$ regions, we apply one of the upper bounds of Proposition~\ref{prop:bound}. This is done in Sections~\ref{sec:asT0pi4},~\ref{sec:asTpi2pi} and~\ref{sec:asTpi4pi2}. In Figure~\ref{fig:decoupage}, the $12$ regions are represented.\par

\begin{figure}
  \begin{center}
    \includegraphics{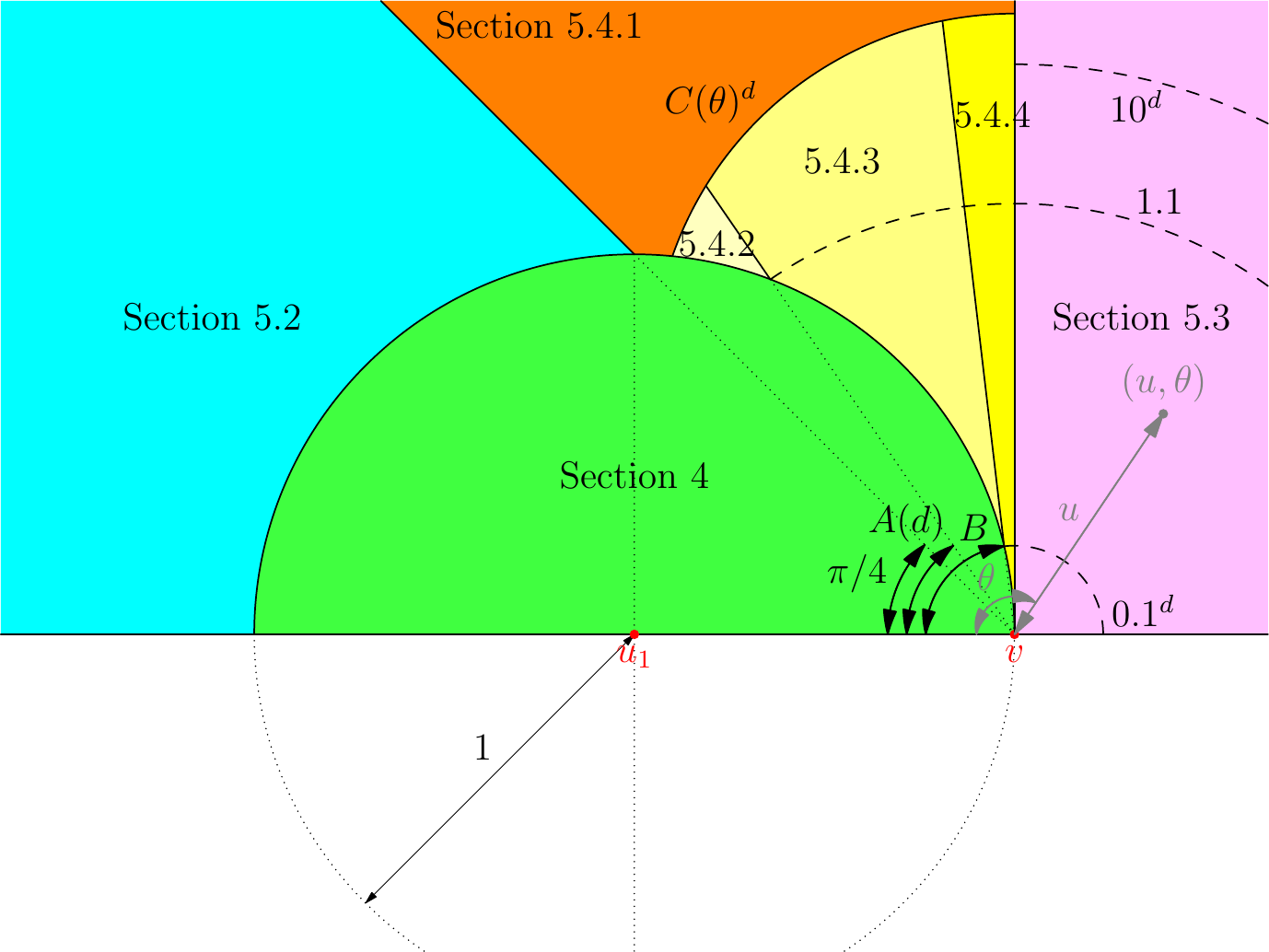}
  \end{center}
  \caption{Each coloured region corresponds to a section or a subsection of this article. The green section corres\-ponds to Section~4 in which we give the exact asymptotic of $T_+(d)$. In others, we give an upper bound for the portion of the integral given $T_-(d)$ that corresponds. The grey point $(u,\theta)$ represents a point of coordinate $(u,\theta)$ on which we integrate. The angles $A(d)$ and $B$ are defined in Section~\ref{sec:asTpi4pi2}. We hope this figure helps to visualise the twelve regions on which we integrate.}  \label{fig:decoupage}
\end{figure}

Before expressing Proposition~\ref{prop:bound}, we introduce the function
\begin{equation}
  C(\theta) = \frac{1}{(\cos \theta) + 0.01}
\end{equation}
that is used all along this section.
\begin{proposition} \label{prop:bound}
  In the following table, for any $d \geq 2$, on the crossing of a line and a column, there is an upper bound of $F\left(\frac{\lune(u^{1/d},\theta)}{V_d}\right) - F(u)$ according to the values of $u$ and $\theta$:
  \begin{center}
    \begin{tabular}{|c||c||c|}
      \hline
       & $\theta \in [\pi/4,\pi/2]$ &  $\theta \in [\pi/2,\pi]$\\ \hline \hline
      $u \in [0,0.1^d]$ & $1/2$ & $1/2$\\ \hline
      $u \in [\max(0.1,2 \cos \theta)^d,1.1]$ & $\frac{18 \sqrt{2}}{\pi \sqrt{d}} \frac{u (\sin \theta)^{d+1}}{(1+u^{2/d}-2u^{1/d}\cos \theta)^{d/2}}$ & $\frac{8\sqrt{2}}{\pi \sqrt{d}} \frac{u(\sin \theta)^{d+1}}{(1+u^{2/d})^{d/2}}$ \\ \hline \hline
      $u \in [1.1,10^d]$ &  & $\frac{2662 \sqrt{2}}{\pi \sqrt{d}} \frac{1}{\left(1+1.1^{2/d}\right)^{d/2}} \frac{\ln u}{u^2}$\\ \hline
      $u \in [10^d,\infty) $ & & $F(u-1)-F(u)$\\ \hline \hline
      $u \in [1.1,C(\theta)^{d}]$ & $\frac{25168 \sqrt{2}}{\pi \sqrt{d}} \frac{\ln u}{u^2} \frac{(\sin \theta)^{d+1}}{\left(1+u^{2/d}-2 u^{1/d} \cos \theta\right)^{d/2}}$ & \\ \hline
      $u \in [C(\theta)^{d},\infty)$ & $F(u-1)-F(u)$ & \\ \hline
    \end{tabular}
  \end{center}
  Moreover, if $\theta \in [0,\pi/4]$ and $u \in [0,\infty)$, an upper bound of $F\left(\frac{\lune(u^{1/d},\theta)}{V_d}\right)-F(u)$ is $F(u-1)-F(u)$.
\end{proposition}

\subsection{Proof of Proposition~\ref{prop:bound}}  \label{sec:bound}
The proof of Proposition~\ref{prop:bound} is based on the two following lemmas: one that gives lower bounds for $\displaystyle \frac{\lune(u^{1/d},\theta)}{V_d}$ and the other one that gives upper bounds for $F(u-\epsilon)-F(u)$ when $\epsilon \in [0,u]$.
\begin{lemma} \label{lem:llb}
  \begin{enumerate}
  \item For any $d \geq 2$, for any $u \geq 0$ and $\theta \in [0,\pi]$,
    \begin{displaymath}
      \frac{\lune(u^{1/d},\theta)}{V_d} \geq \max(0,u-1).
    \end{displaymath}
    
  \item Moreover, if $\theta \in [\pi/2,\pi]$ or, if $\theta \in [0,\pi/2]$ and $u \in [(2 \cos \theta)^{d}, C(\theta)^d]$,
    \begin{displaymath}
      \frac{\lune(u^{1/d},\theta)}{V_d} \geq u - \frac{\sqrt{2}}{\pi \sqrt{d}} \frac{u (\sin \theta)^{d+1}}{\left(1+u^{2/d}-2u^{1/d}\cos\theta\right)^{d/2}} \left(\frac{1}{u^{-1/d} - \cos \theta} + \frac{1}{u^{1/d}-\cos \theta} \right).
    \end{displaymath}
  \end{enumerate}
\end{lemma}

\begin{lemma} \label{lem:Fub}
  Let the function $\displaystyle F(x) = \frac{1}{x} \left( \frac{\ln(1+x)}{x} - \frac{1}{x+1} \right)$.
  \begin{enumerate}
  \item For any $u \geq 0$, for any $\epsilon \in [0,u]$, $\displaystyle F(u-\epsilon) - F(u) \leq \frac{1}{2}$.    
  \item For any $u \geq 0$, for any $\epsilon \in [0,u]$, $\displaystyle F(u-\epsilon) - F(u) \leq \frac{2}{3} \epsilon$.
  \item For any $u \geq 1.1$, for any $\epsilon \in [0,1]$,
    $F(u-\epsilon) - F(u) \leq \frac{2\ln u}{u(u-1)^2} \epsilon \leq 242 \frac{\ln u}{u^3} \epsilon$.
  \end{enumerate}
\end{lemma}

\subsubsection{Proof of Lemma~\ref{lem:llb}}
\begin{itemize}
\item First, remark that $\lune(z,\theta) \geq  z^d V_d-V_d$ for any $z \in [0,\infty)$ and any $\theta \in [0,\pi]$. Indeed, it is the volume of a ball of radius $z$ minus some elements of a ball of radius $1$, hence in the worst case, the ball of radius $1$ is entirely contained in the ball of radius $z$. Moreover, it is always non negative.
  
\item Unfortunately, this bound is not enough in general and we need to refine it when $\theta \in [\pi/2,\pi]$, or when $\theta \in [0,\pi/2]$ and $2 \cos \theta \leq z \leq C(\theta)$. For this, we recall Equation~\eqref{eq:lune}
  \begin{displaymath}
    \frac{\lune(z,\theta)}{V_d} = z^d - \frac{V_{d-1}}{V_d} \left( \int_{\frac{1-z\cos \theta}{\sqrt{(1-z\cos \theta)^2 + (z \sin \theta)^2}}}^{1} (1-x^2)^{\frac{d-1}{2}}\, \di x + z^d \int_{-1}^{\frac{\cos \theta-z}{\sqrt{(\cos \theta -z)^2 + (\sin \theta)^2}}} (1-x^2)^{\frac{d-1}{2}}\, \di x\right).
  \end{displaymath}
  
  As $\displaystyle \frac{V_{d-1}}{V_d\sqrt{d}}$ decreases in $d$, we upper bound $\displaystyle \frac{V_{d-1}}{V_d}$ by $\displaystyle \frac{V_1}{V_2 \sqrt{2}} \sqrt{d} = \frac{\sqrt{2d}}{\pi}$. Now let us upper bound both terms in the sum inside the parenthesis by doing the change of variables $v = \sqrt{1-x^2}$.
  
  Hence the left term rewrites, remembering that $1-z\cos \theta \geq 0$ because $z \leq C(\theta) < 1/\cos \theta$,
  \begin{align*}
    \int_{\frac{1-z\cos \theta}{\sqrt{(1-z\cos \theta)^2 + (z \sin \theta)^2}}}^{1} (1-x^2)^{\frac{d-1}{2}}\, \di x  & = \int_{0}^{\frac{z \sin \theta}{\sqrt{1+z^2-2z\cos \theta}}} \frac{v^{d}}{\sqrt{1-v^2}}\, \di v \\
                                                                                                                     & \leq \frac{\sqrt{1+z^2-2z\cos \theta}}{1-z \cos \theta} \int_{0}^{\frac{z \sin \theta}{\sqrt{1+z^2-2z\cos \theta}}} v^{d} \, \di v \\
                                                                                                                     & = \frac{\sqrt{1+z^2-2z\cos \theta}}{1-z \cos \theta} \frac{1}{d+1}  \left(\frac{z \sin \theta}{\sqrt{1+z^2-2z\cos \theta}} \right)^{d+1}\\
                                                                                                                     & = \frac{z^d}{d+1} \frac{\sqrt{1+z^2-2z\cos \theta}}{1/z-\cos \theta}   \left(\frac{\sin \theta}{\sqrt{1+z^2-2z\cos \theta}} \right)^{d+1}.
  \end{align*}
  Similarly, remarking that $\cos \theta - z \leq 0$ because $z \geq z(\theta) \geq \cos \theta$, the right term becomes 
  \begin{displaymath}
    \int_{-1}^{\frac{\cos \theta-z}{\sqrt{(z-\cos \theta)^2 + (\sin \theta)^2}}} (1-x^2)^{\frac{d-1}{2}}\, \di x \leq \frac{1}{d+1} \frac{\sqrt{(z-\cos \theta)^2 + (\sin \theta)^2}}{z-\cos \theta} \left(\frac{\sin \theta}{\sqrt{(z-\cos \theta)^2 + (\sin \theta)^2}} \right)^{d+1}. 
  \end{displaymath}

  We conclude the proof because $1/(d+1) \leq 1/d$. \qed
\end{itemize}

\subsubsection{Proof of Lemma~\ref{lem:Fub}}
\begin{enumerate}
  \item The point 1 holds because $F(x) \in [0,1/2]$ for any $x \in [0,\infty[$.
  
  \item The point 2 holds because $F$ is concave upward and decreases, then, for any $u \geq 0$ and $\epsilon \leq u$,
  \begin{displaymath}
    F(u-\epsilon) - F(u) \leq F(0) - F(\epsilon) \leq |F'(0)| \epsilon = \frac{2}{3} \epsilon.
  \end{displaymath}

  \item Finally, if $u \geq 1.1$ and $\epsilon \leq 1$, then
  \begin{displaymath}
    F(u-\epsilon) - F(u) \leq (F(u-1)-F(u)) \epsilon
  \end{displaymath}
  because $F$ is concave upward and decreases (Thales' theorem). But,
  \begin{align*}
    F(u-1) - F(u) & = \frac{1}{u-1} \left( \frac{\ln(u)}{u-1} - \frac{1}{u} \right) - \frac{1}{u} \left( \frac{\ln(1+u)}{u} - \frac{1}{u+1} \right)\\
                  & = \frac{u^2\ln u - (u-1)^2 \ln (1+u)}{u^2 (u-1)^2} - \left(\frac{1}{u(u-1)} - \frac{1}{u(u+1)} \right)\\
                  & \leq \frac{2 u \ln u}{u^2 (u-1)^2} - \underbrace{\frac{2}{(u-1)u(u+1)}}_{\geq 0} \leq \frac{2 \ln u}{u (u-1)^2} \leq 2 \cdot 11^2\ \frac{\ln u}{u^3}.
  \end{align*}
\end{enumerate}
Remark, we obtain that $u^2\ln u - (u-1)^2 \ln (1+u) \leq 2 u \ln u$ because both $u(u-2) \leq (u-1)^2$ and $0 \leq \ln u \leq \ln (1+u)$. \qed

\subsubsection{Proposition~\ref{prop:bound} from Lemmas~\ref{lem:llb} and~\ref{lem:Fub}}
\begin{proof}[Proof of Proposition~\ref{prop:bound}]
The five cases corresponding to the bounds $1/2$ and $F(u-1)-F(u)$ are obvious respectively by points~$1$ of Lemma~\ref{lem:Fub} and~\ref{lem:llb}.

Now, let's do the four last cases.
\begin{itemize}

\item When $u \in [\max(0.1,2 \cos \theta)^d,1.1]$ and $\theta \in [\pi/4,\pi/2]$: by both points~$2$ of Lemmas~\ref{lem:llb} and~\ref{lem:Fub},
  \begin{displaymath}
    F\left( \frac{\lune(u^{1/d},\theta)}{V_d} \right) - F(u) \leq \frac{2}{3} \frac{\sqrt{2}}{\pi \sqrt{d}} (\sin \theta)^{d+1} \frac{u}{(1+u^{2/d}-2u^{1/d}\cos \theta)^{d/2}} \left(\frac{1}{u^{-1/d}-\cos \theta} + \frac{1}{u^{1/d}-\cos \theta} \right).
  \end{displaymath}
  \begin{itemize}
  \item If $2 \cos \theta \geq 0.1$ and $d \geq 1$, then
    \begin{displaymath}
      \frac{1}{u^{-1/d}-\cos \theta} + \frac{1}{u^{1/d}-\cos \theta}
      \leq \frac{1}{1.1^{-1/d}-\cos \theta} + \frac{1}{2 \cos \theta-\cos \theta}
      \leq \frac{1}{\frac{1}{1.1} - \frac{\sqrt{2}}{2}} + \frac{1}{\cos \theta}
      \leq 5 + 20.
    \end{displaymath}
  \item If $2 \cos \theta < 0.1$ and $d \geq 1$, then
    \begin{displaymath}
      \frac{1}{u^{-1/d}-\cos \theta} + \frac{1}{u^{1/d}-\cos \theta} 
      \leq \frac{1}{1.1^{-1/d}-\cos \theta} + \frac{1}{0.1-\cos \theta} 
      \leq \frac{1}{\frac{1}{1.1} - 0.05} + \frac{1}{0.05} 
      \leq 1.2 + 20.
    \end{displaymath}
  \end{itemize}
  Let us remark that, in the case
  \begin{displaymath}
    \frac{\sqrt{2}}{\pi \sqrt{d}} \frac{u (\sin \theta)^{d+1}}{\left(1+u^{2/d}-2u^{1/d}\cos\theta\right)^{d/2}} \left(\frac{1}{u^{-1/d} - \cos \theta} + \frac{1}{u^{1/d}-\cos \theta} \right) > u,
  \end{displaymath}
  the bound still holds because
  \begin{align}
    & F\left( \frac{\lune(u^{1/d},\theta)}{V_d} \right) - F(u) \leq F\left( 0 \right) - F(u) \leq \frac{2}{3} u \nonumber \\
    & \quad \leq \frac{2}{3} \frac{\sqrt{2}}{\pi \sqrt{d}} \frac{u (\sin \theta)^{d+1}}{\left(1+u^{2/d}-2u^{1/d}\cos\theta\right)^{d/2}} \left(\frac{1}{u^{-1/d} - \cos \theta} + \frac{1}{u^{1/d}-\cos \theta} \right) \label{eq:pb}
  \end{align}
  and we can end in the same way as before.
  
\item If $u \in [0.1^d,1.1]$, $\theta \in [\pi/2,\pi]$ (in particular, $\cos \theta \leq 0$) and $d \geq 1$: by points~$2$ of Lemmas~\ref{lem:llb} and~\ref{lem:Fub},
\begin{align*}
  F\left( \frac{\lune(u^{1/d},\theta)}{V_d} \right) - F(u) & \leq \frac{2}{3} \frac{\sqrt{2}}{\pi \sqrt{d}} (\sin \theta)^{d+1} \frac{u}{(1+u^{2/d}-2u^{1/d}\cos \theta)^{d/2}} \left(\frac{1}{u^{-1/d}-\cos \theta} + \frac{1}{u^{1/d}-\cos \theta} \right)\\
                                                           & \leq \frac{2}{3} \frac{\sqrt{2}}{\pi \sqrt{d}} (\sin \theta)^{d+1} \frac{u}{(1+u^{1/d})^{d/2}} \left(\frac{1}{1.1^{-1}} + \frac{1}{0.1} \right).
\end{align*}
The case
\begin{displaymath}
  \frac{\sqrt{2}}{\pi \sqrt{d}} \frac{u (\sin \theta)^{d+1}}{\left(1+u^{2/d}-2u^{1/d}\cos\theta\right)^{d/2}} \left(\frac{1}{u^{-1/d} - \cos \theta} + \frac{1}{u^{1/d}-\cos \theta} \right) > u
\end{displaymath}
is done via Equation~\eqref{eq:pb}.

\item If $u \in [1.1,10^d]$, $\theta \in [\pi/2,\pi]$ and $d \geq 1$:
  \begin{itemize}
  \item if $\displaystyle \frac{\sqrt{2}}{\pi \sqrt{d}} (\sin \theta)^{d+1} \frac{u}{(1+u^{2/d}-2u^{1/d}\cos \theta)^{d/2}} \left(\frac{1}{u^{-1/d}-\cos \theta} + \frac{1}{u^{1/d}-\cos \theta} \right) \leq 1$,  by point~$2$ of Lemma~\ref{lem:llb} and point~$3$ of Lemma~\ref{lem:Fub},
    \begin{align*}
      F\left( \frac{\lune(u^{1/d},\theta)}{V_d} \right) - F(u) & \leq 242 \frac{\ln u}{u^3} \frac{\sqrt{2}}{\pi \sqrt{d}} \frac{u (\sin \theta)^{d+1}}{(1+u^{2/d}-2u^{1/d}\cos \theta)^{d/2}} \left(\frac{1}{u^{-1/d}-\cos \theta} + \frac{1}{u^{1/d}-\cos \theta} \right) \\
                                                               & \leq \frac{242\sqrt{2}}{\pi \sqrt{d}} \frac{\ln u}{u^2}\frac{1}{\left(1+1.1^{2/d}\right)^{d/2}} \left( \frac{1}{1/10} + 1 \right)\\
                                                               & \leq \frac{2662\sqrt{2}}{\pi \sqrt{d}} \frac{\ln u}{u^2}\frac{1}{\left(1+1.1^{2/d}\right)^{d/2}}.
    \end{align*}    
  \item else, $1 < \frac{\sqrt{2}}{\pi \sqrt{d}} (\sin \theta)^{d+1} \frac{u}{(1+u^{2/d}-2u^{1/d}\cos \theta)^{d/2}} \left(\frac{1}{u^{-1/d}-\cos \theta} + \frac{1}{u^{1/d}-\cos \theta} \right)$, by point~$1$ of Lemma~\ref{lem:llb} and point~$3$ of Lemma~\ref{lem:Fub},
    \begin{align*}
      & F\left( \frac{\lune(u^{1/d},\theta)}{V_d} \right) - F(u) \leq F(u-1) - F(u) \leq 242 \frac{\ln u}{u^3} \\
      & \quad \leq 242 \frac{\ln u}{u^3} \frac{\sqrt{2}}{\pi \sqrt{d}} \frac{u (\sin \theta)^{d+1}}{(1+u^{2/d}-2u^{1/d}\cos \theta)^{d/2}} \left(\frac{1}{u^{-1/d}-\cos \theta} + \frac{1}{u^{1/d}-\cos \theta} \right)
    \end{align*}
    and we conclude in the same way.
  \end{itemize}
  
\item If $u \in [1.1,C(\theta)^d]$, $\theta \in [\pi/4,\pi/2]$ and $d \geq 1$: by the same case distinction as above, we can always obtain
  \begin{displaymath}
    F\left( \frac{\lune(u^{1/d},\theta)}{V_d} \right) - F(u) \leq 242 \frac{\ln u}{u^3} \frac{\sqrt{2}}{\pi \sqrt{d}} \frac{u (\sin \theta)^{d+1}}{(1+u^{2/d}-2u^{1/d}\cos \theta)^{d/2}} \left(\frac{1}{u^{-1/d}-\cos \theta} + \frac{1}{u^{1/d}-\cos \theta} \right).
  \end{displaymath}
  But, now,
  \begin{align*}
    \frac{1}{u^{-1/d}-\cos \theta} + \frac{1}{u^{1/d}-\cos \theta} & \leq \frac{1}{C(\theta)^{-1}-\cos \theta} + \frac{1}{1 - \frac{\sqrt{2}}{2} }\\
                                                                   & \leq   \frac{1}{0.01} + (2+\sqrt{2}) \leq 104. \qedhere
  \end{align*}
\end{itemize}
\end{proof}
\bigskip
Now, the rest of the section consists to integrate these nine upper bounds on the twelve domains drawn in Figure~\ref{fig:decoupage} to prove Equation~\eqref{eq:asT-} in Theorem~\ref{thm:as}.

\subsection{Asymptotic when $\theta \in [0,\pi/4]$} \label{sec:asT0pi4}
Let $\theta \in [0,\pi/4]$, in this case, $2 \cos \theta \geq 1/\cos \theta$. Hence, we need to upper bound the following integral
\begin{align*}
  & \quad \int_{0}^{\pi/4} \di \theta \sin(\theta)^{d-2} \int_{(2 \cos(\theta))^d}^\infty \di u \, F\left(\frac{\lune(u^{1/d},\theta)}{V_d}\right)-F(u)\\
  & \leq \int_{0}^{\pi/4} \di \theta \sin(\theta)^{d-2} \int_{(2 \cos(\theta))^d}^\infty \di u \, F\left(u-1\right)-F(u) \text{ (by the last line of Proposition~\ref{prop:bound})}\\
  & = \int_{0}^{\pi/4} \di \theta \sin(\theta)^{d-2} \int_{(2 \cos(\theta))^{d}-1}^{(2 \cos(\theta))^{d}} F(u) \, \di u \text{ (because $\int_0^\infty F =1 < \infty$)}\\
  & = \int_{0}^{\pi/4} \di \theta \sin(\theta)^{d-2} \left[-\frac{\ln(1+u)}{u} \right]_{(2 \cos(\theta))^{d}-1}^{(2 \cos(\theta))^{d}}\\
  & \leq \int_{0}^{\pi/4} \di \theta \sin(\theta)^{d-2} \frac{d \ln(2 \cos \theta)}{(2 \cos \theta)^{d}-1} \text{ (the negative term given by $u=(2\cos \theta)^d$ is forgotten)}\\
  & \leq \frac{d \ln 2}{\sqrt{2}-1} \int_{0}^{\pi/4} \di \theta \sin(\theta)^{d-2} \text{ (because $\theta \in [0,\pi/4]$)}\\
  & \leq \frac{d \ln 2}{2(\sqrt{2}-1)} \frac{2\pi}{4} \left(\frac{\sqrt{2}}{2}\right)^{d} \text{ (remember that $d \geq 2$)}.
\end{align*}

\subsection{Asymptotic when $\theta \in [\pi/2,\pi]$} \label{sec:asTpi2pi}
The goal is to bound the following integral
\begin{displaymath}
  \int_{\pi/2}^{\pi} \di \theta \sin(\theta)^{d-2} \int_{0}^\infty \di u \, F\left(\frac{\lune(u^{1/d},\theta)}{V_d}\right)-F(u).
\end{displaymath}
For that, we split it into four parts and use on each of them the appropriate upper bound obtained in Proposition~\ref{prop:bound}, see Figure~\ref{fig:decoupage}.

\paragraph{When $u \leq 0.1^d$:}
\begin{displaymath}
  \int_{\pi/2}^{\pi} \di \theta \underbrace{\sin(\theta)^{d-2}}_{\leq 1} \int_{0}^{0.1^d} \di u \, \underbrace{F\left(\frac{\lune(u^{1/d},\theta)}{V_d}\right)-F(u)}_{\leq 1/2}  \leq \frac{\pi}{2}\frac{0.1^d}{2}.
\end{displaymath}

\paragraph{When $0.1^d \leq u \leq 1.1$:}
\begin{align*}
  & \quad \int_{\pi/2}^{\pi} \di \theta \sin(\theta)^{d-2} \int_{0.1^d}^{1.1} \di u \, F\left(\frac{\lune(u^{1/d},\theta)}{V_d} \right) -F(u) \\
  & \leq \frac{8 \sqrt{2}}{\pi \sqrt{d}} \int_{\pi/2}^{\pi} \di \theta \underbrace{\sin(\theta)^{2d-1}}_{\leq 1} \int_{0.1^d}^{1.1} \frac{u}{(1+u^{2/d})^{d/2}}\\
  & \leq \frac{4 \sqrt{2}}{\sqrt{d}} \int_{0.1^d}^{1.1} \left(\frac{u^{2/d}}{1+u^{2/d}}\right)^{d/2}
  \leq \frac{4.4 \sqrt{2}}{\sqrt{d}} \left(\frac{1.1^{2/d}}{1+1.1^{2/d}} \right)^{d/2} \leq \frac{4.4 \sqrt{2}}{\sqrt{d}} \left(\sqrt{\frac{11}{21}}\right)^{d}
\end{align*}
because $x \mapsto x/(1+x)$ increases on $[0,\infty)$ and $d \mapsto 1.1^{2/d}$ decreases on $[2,\infty)$.

\paragraph{When $1.1 \leq u \leq 10^d$:}
\begin{align*}
  & \quad \int_{\pi/2}^{\pi} \di \theta \underbrace{\sin(\theta)^{d-2}}_{\leq 1} \int_{1.1}^{10^d} \di u \, F\left( \frac{\lune(u^{1/d},\theta)}{V_d} \right) -F(u)\\
  & \leq \frac{\pi}{2}\frac{2662 \sqrt{2}}{\pi \sqrt{d}} \frac{1}{(1+1.1^{2/d})^{d/2}} \int_{1.1}^{10^d} \di u \,  \frac{\ln(u)}{u^2}\\
  & \leq \frac{1331 \sqrt{2}}{\sqrt{d}} \left(\sqrt{ \frac{10}{21}} \right)^d \underbrace{\int_{1}^{\infty} \di u \,  \frac{\ln(u)}{u^2}}_{=1}  \\
  & = \frac{1331 \sqrt{2}}{\sqrt{d}} \left( \sqrt{\frac{10}{21}} \right)^d.
\end{align*}

\paragraph{When $10^d \leq u$:}
\begin{align*}
  & \quad \int_{\pi/2}^{\pi} \di \theta \underbrace{\sin(\theta)^{d-2}}_{\leq 1} \int_{10^d}^{\infty} \di u \, F\left( \frac{\lune(u^{1/d},\theta)}{V_d} \right) -F(u)\\
  & \leq \frac{\pi}{2} \int_{10^d}^{\infty} \di u  F( u -1) -F(u)\\
  & = \frac{\pi}{2} \left[-\frac{\ln(x+1)}{x} \right]_{10^d-1}^{10^d}\\
  & \leq \frac{\pi}{2} \frac{d \ln(10)}{10^d-1} \leq \frac{d \pi \ln(10)}{2}  \left(\frac{1}{9} \right)^d.
\end{align*}

\subsection{Asymptotic when $\theta \in [\pi/4,\pi/2]$} \label{sec:asTpi4pi2}
To finish this article, we need to upper bound 
\begin{displaymath}
  \int_{\pi/4}^{\pi/2} \di \theta \sin(\theta)^{d-2} \int_{(2 \cos \theta)^d}^{\infty} \di u \, F\left( \frac{\lune(u^{1/d},\theta)}{V_d} \right) -F(u).
\end{displaymath}

This last integral is split into seven parts according to both $u$ and $\theta$, see Figure~\ref{fig:decoupage} to see these seven parts. We introduce the following two notations: for any $d \geq 2$,
\begin{equation} 
  A(d) = \arccos\left(\frac{1.1^{1/d}}{2} \right) \text{ and } B = \arccos(0.0 5).
\end{equation}
Just remark that $\left(2 \cos(A(d)) \right)^d = 1.1$ and $(2 \cos(B))^d = 0.1^d$, and so $\displaystyle \frac{\pi}{4} \leq A(d) \leq \frac{\pi}{3} \leq B \leq \frac{\pi}{2}$.

\subsubsection{When $\theta \in [\pi/4,\pi/2]$ and $u \geq C(\theta)^d$}
Firstly, remark that in that case then $1 < \frac{1}{0.01+\sqrt{2}/2} \leq C(\theta) \leq 100$.
\begin{align*} 
  & \quad \int_{\pi/4}^{\pi/2} \di \theta \underbrace{\sin(\theta)^{d-2}}_{\leq 1} \int_{C(\theta)^d}^\infty \di u \,F\left(\frac{\lune(u^{1/d},\theta)}{V_d}\right)-F(u) \\
  & \leq \int_{\pi/4}^{\pi/2} \di \theta \ \frac{d\ln(C(\theta))}{C(\theta)^{d}-1} \text{ (see case $u\geq 10^d$ in Section~\ref{sec:asTpi2pi})} \\
  & \leq d \int_{\pi/4}^{\pi/2} \di \theta\ \frac{\ln 100}{1-C(\theta)^{-d}} C(\theta)^{-d} \\
  & \leq \frac{d \ln 100}{1-(0.01+\sqrt{2}/2)^d} \int_{\pi/4}^{\pi/2} \di \theta\ ((\cos \theta)+0.01)^d\\
  & \leq \frac{d \ln 100}{1-(0.01+\sqrt{2}/2)^2} \frac{\pi}{4} \left( \frac{\sqrt{2}}{2}+0.01 \right)^{d}.
\end{align*}

\subsubsection{When $\theta \in [\pi/4,A(d)]$ and $u \in \left[(2 \cos \theta)^d,C(\theta)^{d}\right]$} \label{sec:pi4Ad}
We remark in that case that $u \geq (2 \cos A(d))^d \geq 1.1$. Hence, by Proposition~\ref{prop:bound},
\begin{align*}
  & \int_{\pi/4}^{A(d)} \di \theta \sin(\theta)^{d-2} \int_{(2 \cos \theta)^d}^{C(\theta)^d} \di u \, F\left(\frac{\lune(u^{1/d},\theta)}{V_d}\right)-F(u)\\
  & \leq \frac{25168 \sqrt{2}}{\pi \sqrt{d}} \int_{\pi/4}^{A(d)} \di \theta \sin(\theta)^{2d-1} \int_{(2 \cos \theta)^d}^{C(\theta)^d} \di u \, \frac{\ln u}{u^2} \frac{1}{(1+u^{2/d}-2u^{1/d} \cos \theta)^{d/2}}.
\end{align*}
Now, remarking that $g:x \mapsto (1+x^2-2x \cos \theta)$ is positive and increases on $[\cos \theta,\infty)$, we get the following upper bound
\begin{align*}
  & \leq  \frac{25168 \sqrt{2}}{\pi \sqrt{d}} \int_{\pi/4}^{A(d)} \di \theta \sin(\theta)^{2d-1} \underbrace{\int_{(2 \cos \theta)^d}^{C(\theta)^d} \di u \, \frac{\ln u}{u^2}}_{\leq  \int_{1}^\infty \di u \, \frac{\ln u}{u^2} = 1}  \underbrace{\frac{1}{(1+ 4 (\cos\theta)^2 - 4 (\cos \theta)^2)^{d/2}}}_{=1}\\
  & \leq \frac{25168 \sqrt{2}}{\pi \sqrt{d}} \left( A(d) - \frac{\pi}{4} \right) \left( \sin A(d) \right)^{2d-1}\\
  & \leq  \frac{25168 \sqrt{2}}{\pi\sqrt{d}} \frac{\pi}{12} \frac{2}{\sqrt{3}} \left(\frac{3}{4} \right)^{d} \text{ (because $A(d) \leq \pi/3$)}.
\end{align*}

\subsubsection{When $\theta \in [A(d),B]$} \label{sec:AB}
In that case, $0.1^d \leq (2 \cos \theta)^d \leq 1.1$. We then split the integral on $u$ according to $1.1$,
\begin{align*}
  & \int_{A(d)}^{B} \di \theta \sin(\theta)^{d-2} \int_{(2 \cos \theta)^d}^{C(\theta)^d} \di u \, \left(F\left(\frac{\lune(u^{1/d},\theta)}{V_d}\right)-F(u)\right).
\end{align*}

\paragraph{When $u \in [1.1,C(\theta)^d]$:}
By Proposition~\ref{prop:bound},
 \begin{align*}
   & \int_{A(d)}^{B} \di \theta \sin(\theta)^{d-2} \int_{1.1}^{C(\theta)^d} \di u \, \left(F\left(\frac{\lune(u^{1/d},\theta)}{V_d}\right)-F(u)\right)\\
   & \leq \frac{25168 \sqrt{2}}{\pi \sqrt{d}} \int_{A(d)}^{B} \di \theta \sin(\theta)^{2d-1} \int_{1.1}^{C(\theta)^d} \di u \, \frac{\ln u}{u^2} \frac{1}{\left(1+u^{2/d} - 2 u^{1/d} \cos \theta\right)^{d/2}}.
 \end{align*}
 Now, because $g:x \mapsto (1+x^2-2x \cos \theta)$ is positive and increases on $[\cos \theta,\infty)$ and that $\cos \theta \leq 1 \leq 1.1^{1/d} \leq u^{1/d}$, then $1+u^{2/d} - 2 u^{1/d} \cos \theta = g(u^{1/d}) \geq g(1) = 2(1-\cos \theta)$, and so
 \begin{align*}
   & \int_{A(d)}^{B} \di \theta \sin(\theta)^{d-2} \int_{1.1}^{C(\theta)^d} \di u \, \left(F\left(\frac{\lune(u^{1/d},\theta)}{V_d}\right)-F(u)\right)\\
   & \leq \frac{25168 \sqrt{2}}{\pi \sqrt{d}} \frac{1}{2^{d/2}} \int_{A(d)}^{B} \di \theta \underbrace{\sin(\theta)^{-1}}_{\leq 1/\sin(\pi/4) = \sqrt{2}} \frac{(\sin \theta)^{2d}}{(1-\cos \theta)^{d/2}} \underbrace{\int_{1.1}^{C(\theta)^d} \di u \, \frac{\ln u}{u^2}}_{\leq 1}\\
   & \leq \frac{50336}{\pi \sqrt{d}} \frac{1}{2^{d/2}} \int_{A(d)}^{B} \di \theta \left(\frac{(\sin \theta)^{4}}{1-\cos \theta}\right)^{d/2}.
 \end{align*}
 Now, remark that $\frac{(\sin \theta)^{4}}{1-\cos \theta} \leq \frac{32}{27}$. A simple way is to see that the maximum of $\theta \mapsto \frac{(\sin \theta)^{4}}{1-\cos \theta}$ on $[0,\pi/2]$ is the same as the one of $x \mapsto (1-x)(1+x)^2$ on $[0,1]$ setting $x = \cos t$. Then
 \begin{align*}
   & \int_{A(d)}^{B} \di \theta \sin(\theta)^{d-2} \int_{1.1}^{C(\theta)^d} \di u \, \left(F\left(\frac{\lune(u^{1/d},\theta)}{V_d}\right)-F(u)\right)\\
   & \leq \frac{50336}{\pi \sqrt{d}} \frac{1}{2 ^{d/2}} (B-A(d)) \left(\frac{32}{27} \right)^{d/2} \leq  \frac{50336}{\pi \sqrt{d}} \frac{\pi}{4} \left( \frac{4 \sqrt{3}}{9} \right)^d.
 \end{align*}
 
\paragraph{When $u \in [(2 \cos \theta)^d,1.1]$:} By Proposition~\ref{prop:bound},
\begin{align*}
  & \quad \int_{A(d)}^{B} \di \theta \sin(\theta)^{d-2} \int_{(2 \cos \theta)^d}^{1.1} \di u \, \left(F\left(\frac{\lune(u^{1/d},\theta)}{V_d}\right)-F(u)\right) \\
  & \leq \frac{18\sqrt{2}}{\pi \sqrt{d}} \int_{A(d)}^B \di \theta \sin(\theta)^{2d-1} \int_{(2 \cos \theta)^d}^{1.1} \di u \, \left( \frac{u^{2/d}}{1+u^{2/d}-2u^{1/d} \cos \theta}\right)^{d/2}.
\end{align*}
The function $\frac{x^2}{1+x^2-2x \cos \theta}$ is positive and increases on $[0,1/\cos \theta]$, and so on $[0,1.1]$. Indeed, $1/\cos \theta > 2 \cdot 1.1^{-1/d} > 1.1$ because  $d \geq 2$, then
\begin{align*} 
  & \int_{A(d)}^{B} \di \theta \sin(\theta)^{d-2} \int_{(2 \cos \theta)^d}^{1.1} \di u \, \left(F\left(\frac{\lune(u^{1/d},\theta)}{V_d}\right)-F(u)\right) \\
  & \leq \frac{19.8\sqrt{2}}{\pi \sqrt{d}} \int_{A(d)}^B \di \theta \sin(\theta)^{2d-1} \left( \frac{1.1^{2/d}}{1+1.1^{2/d}-2\times 1.1^{1/d} \cos \theta}\right)^{d/2}\\
  & \leq \frac{19.8\sqrt{2}}{\pi \sqrt{d}} \int_{A(d)}^B \di \theta \sin(\theta)^{2d-1} \left( \frac{1.1^{1/d}}{1.1^{-1/d}+1.1^{1/d}-2 \cos \theta}\right)^{d/2}.
\end{align*}
The last new argument is that the function $x \mapsto x+1/x$ is minimal when $x=1$. Then, we conclude as in the previous case $u \in [1.1,C(\theta)^d]$:
\begin{align*}
  & \leq \frac{19.8\sqrt{2}}{\pi \sqrt{d}} \sqrt{1.1} \frac{1}{\sin(\pi/4)} \int_{A(d)}^B \di \theta \left( \frac{\sin(\theta)^{4} }{2-2 \cos \theta}\right)^{d/2}\\
  & \leq \frac{19.8 \sqrt{2.2}}{\pi \sqrt{d}} \sqrt{2} \frac{1}{2^{d/2}} \int_{A(d)}^B \di \theta \left(\frac{(\sin \theta)^4}{1-\cos \theta}\right)^{d/2}\\
  & \leq \frac{39.6 \sqrt{1.1}}{\pi \sqrt{d}} \frac{1}{2^{d/2}} \left(B-A(d) \right) \left( \frac{32}{27}\right)^{d/2} \leq \frac{39.6 \sqrt{1.1}}{\pi \sqrt{d}} \frac{\pi}{4} \left(\frac{4\sqrt{3}}{9} \right)^d.
\end{align*}

\subsubsection{When $\theta \in [B,\pi/2]$}\label{sec:Bpi2}
\paragraph{When $u \in [1.1,C(\theta)^{d}]$:}
The first computation from Section~\ref{sec:AB} holds and like $\pi/2-B \leq \pi/4$,
\begin{displaymath}
  \int_{B}^{\pi/2} \di \theta \sin(\theta)^{d-2} \int_{1.1}^{C(\theta)^d} \di u \, F\left(\frac{\lune(u^{1/d},\theta)}{V_d}\right)-F(u) \leq  \frac{50336}{\pi\sqrt{d}} \frac{\pi}{4} \left(\frac{4\sqrt{3}}{9} \right)^{d}.
\end{displaymath}

\paragraph{When $u \in [0.1^d,1.1]$:}
The second computation from Section~\ref{sec:AB} holds, then
\begin{displaymath}
  \int_{B}^{\pi/2} \di \theta \sin(\theta)^{d-2} \int_{0.1^d}^{1.1} \di u \, F\left(\frac{\lune(u^{1/d},\theta)}{V_d}\right)-F(u) \leq \frac{39.6 \sqrt{1.1}}{\pi \sqrt{d}} \frac{\pi}{4} \left(\frac{4\sqrt{3}}{9} \right)^d.
\end{displaymath}

\paragraph{When $u \in [(2 \cos \theta)^d,0.1^d]$:}
\begin{displaymath}
  \int_{B}^{\pi/2} \di \theta \underbrace{\sin(\theta)^{d-2}}_{\leq 1} \int_{(2 \cos \theta)^d}^{0.1^d} \di u \, \underbrace{F\left(\frac{\lune(u^{1/d},\theta)}{V_d}\right)-F(u)}_{\leq 1/2} \leq \left( \frac{\pi}{2} - B \right) \frac{0.1^d}{2} \leq \frac{\pi}{4} 0.1^d.
\end{displaymath}

\subsection{End of the proof}
Putting all of the twelve previous asymptotics together and remarking that $\left(\frac{4\sqrt{3}}{9} \right)^d$ is the exponential dominating term, we obtain that
\begin{displaymath}
  T_-(d) = O \left( \left(\frac{4 \sqrt{3}}{9}\right)^d \right). \qed
\end{displaymath}

\section{Perspectives}
Our main theorem permits, given a $d$-NNT, to distinguish (asymptotically) its dimension by computing its mean number of siblings.

The number of siblings depends only on the local limit of the NNT. Hence, our main theorems could be generalised and stated for any $d$-dimensional manifold. In particular, we could not distinguish a NNT on $(\mathbb{S}_d,\lVert.\rVert_2,\lambda)$ and a NNT on $(\mathbb{T}_d,\lVert.\rVert_2,\lambda)$ where $\mathbb{T}_d$ denotes the torus in dimension $d$ by computing the mean number of siblings. Hence, it is an open question to distinguish a circle, a segment and a star with $k$ edges.

Another open question is whether could we distinguish a NNT on $(\mathbb{S}_d,\delta,\lambda)$ according to the metric $\delta$, either by the number of siblings or by another statistic. In particular, the number of siblings of NNT on $(\mathbb{S}_d,\lVert.\rVert_p,\lambda)$ and $(\mathbb{S}_d,\lVert.\rVert_{q},\lambda)$ when $\displaystyle \frac{1}{p}+\frac{1}{q}=1$ should be equal because their local limit are the same up to a rotation and a dilation (and so the closest point does not change). But, does the number of siblings depend on $p$ and $q$ if they are not Hölder conjugates?

Finally, with Theorems~\ref{thm:main} and~\ref{thm:as}, we should construct a statistical test that permits to find the dimension $d$ from a $d$-NNT. But this test needs an exponential number of nodes according to dimension~$d$. Could we find a better statistic? 

\section*{Acknowledgement}
The author would like to thank Nicolas Curien to introduce him to the subject of NNT and RRT and to suggest him to compute the mean number of siblings, Alice Contat for the help during the exact computation in dimension $1$, and Robin Stephenson to introduce him to the literature on geometric preferential attachment graphs and FKP network models. The author would also like to thank ERC 740943 GeoBrown and ANR RanTanPlan for their support.

\bibliographystyle{alpha}
\bibliography{aaa}

\newcommand{\etalchar}[1]{$^{#1}$}
\begin{thebibliography}{BBB{\etalchar{+}}03}

\bibitem[Ald18]{Aldous18}
David Aldous.
\newblock Random partitions of the plane via {P}oissonian coloring and a
  self-similar process of coalescing planar partitions.
\newblock {\em The Annals of Probability}, 46(4):2000--2037, 2018.

\bibitem[BB14]{BB14}
Erich Baur and Jean Bertoin.
\newblock Cutting edges at random in large recursive trees.
\newblock In {\em Stochastic Analysis and Applications 2014: In Honour of Terry
  Lyons}, pages 51--76. Springer, 2014.

\bibitem[BBB{\etalchar{+}}03]{BBBCR03}
Noam Berger, B{\'e}la Bollob{\'a}s, Christian Borgs, Jennifer Chayes, and
  Oliver Riordan.
\newblock Degree distribution of the fkp network model.
\newblock In {\em International Colloquium on Automata, Languages, and
  Programming}, pages 725--738. Springer, 2003.

\bibitem[BBCS23]{BBCS23}
Anne-Laure Basdevant, Guillaume Blanc, Nicolas Curien, and Arvind Singh.
\newblock Fractal properties of the frontier in poissonian coloring.
\newblock {\em arXiv preprint arXiv:2302.07254}, 2023.

\bibitem[Cur23]{Curien23}
Nicolas Curien.
\newblock Random graphs : quodlibet.
\newblock
  https://www.imo.universite-paris-saclay.fr/~nicolas.curien/enseignement.html,
  2023.

\bibitem[Dev87]{Devroye87}
Luc Devroye.
\newblock Branching processes in the analysis of the heights of trees.
\newblock {\em Acta Informatica}, 24(3):277--298, 1987.

\bibitem[Dev88]{Devroye88}
Luc Devroye.
\newblock Applications of the theory of records in the study of random trees.
\newblock {\em Acta Informatica}, 26(1-2):123, 1988.

\bibitem[DFF10]{DFF10}
Luc Devroye, Omar Fawzi, and Nicolas Fraiman.
\newblock The height of scaled attachment random recursive trees.
\newblock In {\em Discrete Mathematics and Theoretical Computer Science}, pages
  129--142. Discrete Mathematics and Theoretical Computer Science, 2010.

\bibitem[DL95]{DL95}
Luc Devroye and Jiang Lu.
\newblock The strong convergence of maximal degrees in uniform random recursive
  trees and dags.
\newblock {\em Random Structures \& Algorithms}, 7(1):1--14, 1995.

\bibitem[Dob96]{Dobrow96}
Robert~P. Dobrow.
\newblock On the distribution of distances in recursive trees.
\newblock {\em Journal of Applied Probability}, 33(3):749--757, 1996.

\bibitem[Drm09]{Drmota09}
Michael Drmota.
\newblock {\em Random trees: an interplay between combinatorics and
  probability}.
\newblock Springer Science \& Business Media, 2009.

\bibitem[FK23]{FK23}
Alan Frieze and Micha{\l} Karo{\'n}ski.
\newblock {\em Introduction to random graphs}.
\newblock Cambridge University Press, 2023.

\bibitem[FKP02]{FKP02}
Alex Fabrikant, Elias Koutsoupias, and Christos~H Papadimitriou.
\newblock Heuristically optimized trade-offs: A new paradigm for power laws in
  the internet.
\newblock In {\em Automata, Languages and Programming: 29th International
  Colloquium, ICALP 2002 M{\'a}laga, Spain, July 8--13, 2002 Proceedings 29},
  pages 110--122. Springer, 2002.

\bibitem[GM05]{GM05}
Christina Goldschmidt and James Martin.
\newblock {Random Recursive Trees and the Bolthausen-Sznitman Coalesent}.
\newblock {\em Electronic Journal of Probability}, 10:718 -- 745, 2005.

\bibitem[GS02]{GS02}
William Goh and Eric Schmutz.
\newblock Limit distribution for the maximum degree of a random recursive tree.
\newblock {\em Journal of Computational and Applied Mathematics}, 142(1):61 --
  82, 2002.
\newblock Probabilistic Methods in Combinatorics and Combinatorial
  Optimization.

\bibitem[Jor10]{Jordan10}
Jonathan Jordan.
\newblock Degree sequences of geometric preferential attachment graphs.
\newblock {\em Advances in Applied Probability}, 42(2):319--330, 2010.

\bibitem[JW15]{JW15}
Jonathan Jordan and Andrew~R Wade.
\newblock Phase transitions for random geometric preferential attachment
  graphs.
\newblock {\em Advances in Applied Probability}, 47(2):565--588, 2015.

\bibitem[LM21]{LM21}
Lyuben Lichev and Dieter Mitsche.
\newblock New results for the random nearest neighbor tree.
\newblock {\em arXiv preprint arXiv:2108.13014}, 2021.

\bibitem[Mah91]{Mahmoud91}
Hosam~M Mahmoud.
\newblock Limiting distributions for path lengths in recursive trees.
\newblock {\em Probability in the Engineering and Informational Sciences},
  5(1):53--59, 1991.

\bibitem[MS02]{MS02}
Subhrangshu~S Manna and Parongama Sen.
\newblock Modulated scale-free network in euclidean space.
\newblock {\em Physical Review E}, 66(6):066114, 2002.

\bibitem[Pit94]{Pittel94}
Boris Pittel.
\newblock Note on the heights of random recursive trees and random m-ary search
  trees.
\newblock {\em Random Structures \& Algorithms}, 5(2):337--347, 1994.

\bibitem[Pre09]{Preater09}
John Preater.
\newblock A species of voter model driven by immigration.
\newblock {\em Statistics \& probability letters}, 79(20):2131--2137, 2009.

\bibitem[PS22]{PS22}
Michel Pain and Delphin S{\'e}nizergues.
\newblock Correction terms for the height of weighted recursive trees.
\newblock {\em The Annals of Applied Probability}, 32(4):3027--3059, 2022.

\bibitem[SM95]{SM95}
Robert~T Smythe and Hosam~M Mahmoud.
\newblock A survey of recursive trees.
\newblock {\em Theory of Probability and Mathematical Statistics}, 51(1-27),
  1995.

\end{thebibliography}

\end{document}